%% file: CoarseGTSurvey-Combinatorica-Revision_2.tex
\documentclass[a4paper]{article}
\usepackage{amsthm,amssymb,amsmath,enumerate,graphicx,epsf}

\newcommand{\COLORON}{1}
\newcommand{\NOTESON}{0}
\newcommand{\Debug}{0}
\usepackage{bbold}
\input{defs}

\newcommand{\andim}{\ensuremath{\mathrm{ANdim}}}

\newcommand{\forb}[1]{\mathrm{Forb}(#1)}

\newcommand{\rank}{\mathrm{Rank}}
\newcommand{\diam}{\mathrm{diam}}
\newcommand{\ball}{\mathrm{Ball}}

\newcommand{\mb}{meta-bridge}
\newcommand{\game}{$\gamma$--edge}
\newcommand{\asmi}{asymptotic minor}

\newtheorem{remark}{Remark}
\usepackage{mathtools,enumitem}
\usepackage[percent]{overpic}

\usepackage{authblk}
\usepackage[greek,english]{babel}  
\usepackage{hyperref}
\RequirePackage{latexsym} \RequirePackage{amsthm}
\RequirePackage{amsmath}

\usepackage{enumerate}
\RequirePackage{amssymb} \RequirePackage{makeidx}
\usepackage{dsfont}
\usepackage{mathrsfs}

\newtheorem{Def}{Definition}

\newcommand{\ncm}{near-component}

\newcommand{\pp}[1]{ \ifnum \Debug = 1 {\marginpar{\tiny #1 --pp}} \fi}

\newcommand{\asm}[2]{\ensuremath{#1 \prec^\infty #2}}
\newcommand{\nasm}[2]{\ensuremath{#1 \not \prec^\infty #2}}

\newcommand{\asdim}{\mathrm{asdim}}
\newcommand{\fan}{\mathrm{Fan}}

\usepackage{authblk}

\begin{document}

\title{Graph minors and metric spaces}

\author[1]{Agelos Georgakopoulos\thanks{Supported by the European Research Council (ERC) under the European Union's Horizon 2020 research and innovation programme (grant agreement No 639046), and EPSRC grants EP/V048821/1 and EP/V009044/1.}}
\affil[1]{  {Mathematics Institute}\\ {University of Warwick}\\  {CV4 7AL, UK}}
\author[2]{Panos Papasoglu}
\affil[2]{
{Mathematical Institute,
University of Oxford,
Andrew Wiles Building, 
Radcliffe Observatory Quarter, 
Woodstock Road, 
Oxford, 
OX2 6GG, 
UK\\
papazoglou@maths.ox.ac.uk  }}

\date{\today}
\maketitle

\begin{abstract}
We present problems and results that combine graph-minors and coarse geometry. For example, we ask whether every geodesic metric space (or graph) without a fat $H$ minor is quasi-isometric to a graph with no $H$ minor, for an arbitrary finite graph $H$. We answer this affirmatively for a few small $H$. We also present a metric analogue of Menger's theorem and K\"onig's ray theorem. We conjecture metric analogues of the \Erd--P\'osa Theorem and Halin's grid theorem.
\end{abstract}



\maketitle

\section{Introduction}

Gromov's \cite{GroAsyInv} coarse geometry perspective had a groundbreaking impact on geometric group theory, also spreading onto nearby areas. As an example highlight, it has been proved that (any \Cg\ of) every finitely generated group  with finite asymptotic dimension embeds coarsely into Hilbert space \cite{YuNov}, and therefore it satisfies the  Novikov conjecture \cite{GroAsyInv}. 

Several recent papers study the interplay between Gromov's asymptotic dimension and graph-theoretic notions such as graph minors \cite{quasitrees,BBEGLPS,FujPapCoa,FujPapAsy,JorLanGeo,OstExp,OstRosMet}. With these in the background, we present some results and questions with a strong interplay between geometry and graph theory. We hope that this will evolve into a coherent theory that could be called \defi{`Coarse Graph Theory'} or \defi{`Graph-Theoretic Geometry'}.

We would like to view graphs `from far away' to see their large-scale geometry and its implications. More generally, the same idea can be applied to Riemannian manifolds, or more generally, arbitrary length spaces. In what follows $X$ can stand for a length space, or an infinite graph, or a family of finite graphs, endowed with their graph-distance to turn them into metric spaces; none of our statements are affected by this choice, the reader may choose their preferred setting. Our favourite such problem is

\begin{conjecture} \label{conj fat min}
Let $X$ be a graph or a length space, and let $H$ be a finite graph. Then $X$ has no $K$-fat $H$ minor for some $K\in \N$  \iff\ $X$ is $f(K)$-quasi-isometric to a graph with no $H$ minor, where $f: \N\to \N$ depends on $H$ only.\footnote{The backward direction is straightforward.}
\end{conjecture}

A quasi-isometry is a generalisation of a bi-Lipschitz map that allows for an additive error; see \Sr{sec MS}. This notion lies at the heart of Gromov's programme \cite{GroAsyInv}. A  fat minor is a coarse version of the notion of a graph minor  defined as follows.
Recall that $X$ has a (usually finite) graph $H$ as a \defi{minor}, if we can find connected \defi{branch sets} and \defi{branch paths} in $X$, \st\ after contracting each branch set to a point, and each branch path to an edge, we obtain a copy of $H$. We say that $X$ has $H$ as a \defi{$K$-fat minor}, if we can choose the above sets to be pairwise at distance $\geq K$, except that we do not require this for incident branch set--branch edge pairs ---which must be at distance 0 by definition; see also \Dr{def fat} below. This notion of fat minor also features in \cite{BBEGLPS}, and a slight variation features in \cite{CDNRV}.

If $X$ has a $K$-fat $H$ minor for every $K$, then we say that $X$ has $H$ as an \defi{\asmi}, and write \defi{$\asm{H}{X}$}. An important feature of asymptotic minors is that they are preserved under quasi-isometries, and in particular they are stable under changing the generating set in a \Cg\ of a group (\Or{invariance}).

For example, let $L$ be the square lattice on $\Z^2$, and let $L':= L \times K_2$ be its cartesian product with an edge. Then one can easily show that $L'$ has every finite graph as a minor, but it does not have $K_5$ or $K_{3,3}$ as a 2-fat minor. Thus this notion correctly captures the intuition that $L'$ looks like the plane from a coarse point of view. As another example, the cubic lattice on $\Z^3$ has every finite graph $H$ as an \asmi; to see this, embed $H$ in $\R^3$, fatten it slightly, then zoom $\Z^3$ out while keeping your $H$ fixed. 

\medskip

We know that \Cnr{conj fat min} holds for some very small graphs $H$. If $H=K_2$ (a pair of vertices with an edge), then  \Cnr{conj fat min} becomes the obvious fact that if $X$ has finite diameter then it is quasi-isometric to a point. The case $H=K_3$ is already non trivial and it follows from Manning's characterization of quasi-trees
\cite{Manning}. We give a self contained exposition of this
which entails a shorter proof of Manning's result (\Tr{triangle-free}). Further alternative characterisations of quasi-trees have been recently obtained by Berger \& Seymour \cite{BerSeyBou}. 

Chepoi, Dragan, Newman, Rabinovich \& Vax\`es \cite{CDNRV} settled the case $H=K_{2,3}$ of \Cnr{conj fat min} using slightly different terminology.  The similar case $H=K_4^-$, i.e.\ $K_4$ with an edge removed, has been settled by Fujiwara and the second author \cite{FujPapCoa}, along with a stronger result characterising quasi-cacti.  
In \Sr{sec stars} we settle the case where $H$ is a star $K_{1,m}, m\geq 1$. This is, as far as we know, a new result.

Call two graphs $H_1,H_2$ \defi{asymptotically equivalent}, if $\asm{H_1}{X}$ implies $\asm{H_2}{X}$ and vice-versa \fe\ space $X$. In \Sr{sec homeo} we show that if two finite graphs $H_1,H_2$ are homeomorphic as 1-complexes, then they are asymptotically equivalent (the converse is true as well if we allow $X$ to be disconnected). This immediately settles \Cnr{conj fat min} for paths and cycles $H$. Moreover, combined with the aforementioned result of \cite{FujPapCoa}, it implies that a graph is quasi-isometric to a cactus \iff\ it is quasi-isometric to an outerplanar graph (\Cr{cor OP}), a fact previously proved by Chepoi et al.\ \cite{CDNRV}. 

\pp{I think the converse IS true, for graphs with no cut vertices or cut edges, so maybe we should think a bit before stating the question. Also the converse IS NOT true is general eg take O and figure 8.}

In \Cnr{conj fat min} we are mostly interested in the case where $H$ is finite, but we point out some infinite cases of interest: in \Sr{sec Konig} we prove a coarse analogue of K\"onig's theorem, and in \Sr{sec Halin} we discuss coarse analogues of  Halin's grid theorem.
\medskip

We observe below (\Or{obs distortion}) that when $H$ is 2-connected in \Cnr{conj fat min}, then instead of a quasi-isometry we can equivalently ask for a bi-Lipschitz map. Bi-Lipschitz embeddings, aka.\ embeddings of bounded multiplicative distortion, into various targets, most notably tree metrics and normed spaces, are an actively researched topic, see e.g.\ \cite{CDNRV,eriksson-bique,NaorInt,Ostrovskii} and references therein. Embeddings of bounded additive distortion are also of interest \cite{CDNRV}. This raises:
\begin{question}
For which finite graphs $H$ is it true that if a graph $X$ has no $K$-fat $H$ minor for some $K\in \N$, then $X$ admits a map of bounded additive distortion onto a graph with no $H$ minor?
\end{question}
We know that this is the case when $H$ is a cycle \cite{KerTre} (see also claim \eqref{fourteen} in \Sr{quasitree}) or $H=K_{1,3}$. 

\medskip
Ostrovskii \& Rosenthal \cite{OstRosMet} proved that (infinite, locally finite) graphs with a forbidden minor have finite asymptotic dimension. Bonamy et al.\ \cite{BBEGLPS} improve this by proving that such graphs have asymptotic dimension at most 2 (as was conjectured in \cite{FujPapCoa}). They ask whether forbidding a fat minor implies bounded asymptotic dimension. Note that a positive answer to \Cnr{conj fat min} would imply this.  The following question is similar in spirit.

\begin{conjecture}[Coarse Hadwiger conjecture] \label{coa Had}
\Fe\ \nin, if $X$ has no asymptotic $K_{n+1}$ minor, then $X$ has Assouad-Nagata dimension $\andim(X)$ at most $n-1$.
\end{conjecture}

Since $\andim(X)$ can be defined via colouring (\Sr{sec asdim})---and $\andim(X)=n-1$ means that $n$ colours suffice--- \Cnr{coa Had}  can be thought of as the coarse version of the Hadwiger conjecture, perhaps the oldest well-known open problem of graph theory, which asserts that every graph with no $K_{n+1}$ minor is $n$--colourable. If we drop the word `asymptotic' in the  statement of \Cnr{coa Had}, then a better bound  $\andim(X)\leq 2$ was recently obtained by Distel \cite{DisPro} and by Liu \cite{LiuAss}.  In fact 2 might be the right upper bound in \Cnr{coa Had} as well.  A somewhat related result relaxes the Hadwiger conjecture by allowing monochromatic components of bounded size \cite{DEMW}.

Apart from its theoretic interest due to the connection with  the Hadwiger conjecture, \Cnr{coa Had} is also interesting from an algorithmic perspective. Given the theoretical and algorithmic importance of minor-closed graph classes, one hopes that the graphs we encounter in practice come from such a class. This is however rarely the case, because e.g.\ a little bit of randomness will result in arbitrarily large cliques. Thus it is natural to seek extensions of the theory of graph minors to a larger setting. This has been a mainstream trend in computer science, see e.g.\ \cite{GrKrSiDec} and references therein. \Cnr{coa Had}  pursues something similar in a different direction. The kind of practical problems we have in mind are those typically encountered in computational geometry, e.g.\ the Traveling Salesman Problem (TSP). Computer scientists have developed a notion called \defi{sparse partition schemes}, and have shown how these can be used to solve important practical problems such as the Universal Steiner Tree and the Universal TSP problem. The definition of sparse partition schemes is a slightly more quantitative version of the colourings used to define the Assouad-Nagata dimension (\Sr{sec asdim}); see \cite{BBEGLPS} and references therein for the precise relationship and further applications).
Recall that the presence of large cliques in real-world graphs is an obstruction to using algorithms coming from graph minor theory. Our coarse approach overcomes such obstructions: a clique of any size has diameter 1, and so it looks just like a point when seen from far away. 

In the hope to develop more basic tools towards \Cnr{conj fat min} and \Cnr{coa Had}, we considered the following coarse version of Menger's theorem:

\begin{conjecture}[Coarse Menger Theorem] \label{conj MM}
For every $\nin$, there is a constant $C_n$, \st\ the following holds \fe\ graph $G$ and every two subsets  $A,Z$ of its vertex set. \Fe\ $r\in \Z$, either there is a set of $n$ \pths{A}{Z}\ in \g at distance at least  $r$ from each other, 
or there is a set $S$ of less than $n$ vertices of \G, \st\ removing the ball of radius $C_n r$ around $S$ separates $A$ from $Z$ in \G.
\end{conjecture}
\pp{changed the footnote-I wonder if their gap was serious or not, or whether we should mention it-their constant now is better anyway. My initial proof had a gap too, that's why I wonder if they made the same mistake.}
In \Sr{sec Menger} we will prove the special case $n=2$ of this conjecture\footnote{An alternative proof was independently obtained by Albrechtsen et al.\ \cite{AHJKW}, who also state \Cnr{conj MM}. 
}  We warn the reader that \Cnr{conj MM} seems very difficult even for $n=3$. In \Sr{sec OP} we discuss related problems that might be more accessible.

\comment{
\begin{conjecture} \label{conj MM NPc}
Fix $k\in \N$, and let \g be a graph on $n$ vertices, and   $A,Z$  two subsets of its vertex set. 

there is a constant $C_n$, \st\ the following holds \fe\ graph $G$  \Fe\ $r\in \Z$, either there is a set of $n$ \pths{A}{Z}\ in \g at distance at least  $r$ from each other, 
or there is a set $S$ of less than $n$ vertices of \G, \st\ removing the ball of radius $C_n r$ around $S$ separates $A$ from $Z$ in \G.
\end{conjecture}
}

\section{Preliminaries}


\subsection{Metric spaces} \label{sec MS}

Let $(X,d)$ be a metric space. Given $x\in X$ and $R\in \R_+$, we let $\ball_x(R)$ denote the ball of radius $R$ around $x$.

For $Y,Z\subseteq X$ we define $d(Y,Z):= \inf_{y\in Y, z\in Z} d(y,z)$.

A \defi{length space} is a metric space $(X,d)$ \st\ for every $x,y\in X$ and $\epsilon >0$ there is an $x$--$y$~arc of length at most $d(x,y)+\epsilon$. Thus every geodesic metric space, in particular every graph endowed with its standard metric, is a length space.

An  \defi{$(M,A)$-quasi-isometry} between graphs $G=(V,E)$ and $H=(V',E')$ is a map\\ $f: V \to V'$ \st\ the following hold for fixed constants $M\geq 1, A\geq 0$:
\begin{enumerate}
\item \label{q i} $M^{-1} d(x,y) -A \leq d(f(x),f(y))\leq M d(x,y) +A $ \fe\ $x,y \in V$, and
\item \label{q ii}  \fe\ $z\in V'$ \ti\ $x\in V$ \st\ $d(z,f(x))\leq A$.
\end{enumerate}
Here $d(\cdot,\cdot)$ stands for the graph distance in the corresponding graph $G$ or $H$. We say that $G$ and $H$ are \defi{quasi-isometric}, if such a map $f$ exists. Quasi-isometries between arbitrary metric spaces are defined analogously. The following is well-known, see e.g.\ {\cite[Proposition~3.B.7(6)]{cornulier_metric_2016}}:

\begin{observation} \label{qi graph}
Every length space is quasi-isometric to a connected graph.
\end{observation}
\begin{proof}[Proof (sketch).]
Given a length space $(X,d)$, fix $\epsilon>0$ and pick a subset $U$ \st\ $d(x,y)>\epsilon$ \fe\ $x,y\in U$ and $d(z,U)\leq \epsilon$ \fe\ $z\in X$. Define a graph on $U$ by joining any two points $x,y\in U$ satisfying $d(x,y)< 3\epsilon$ with an edge. 
\end{proof}

When $f: X \to Y$ between metric spaces satisfies \ref{q i} 
with $A=0$, then we say that $f$ has bounded \defi{multiplicative distortion}. When $M=1$, we say that $f$ has bounded \defi{additive distortion}. We observe that in many cases where $Y$ is not fixed, but allowed to vary inside a class of graphs, then a quasi-isometry can be turned into a map of bounded multiplicative distortion: 
\begin{observation} \label{obs distortion}
Let $\ch$ be a set of 2-connected graphs, and let $\cc=\forb{\ch}$ be the class of graphs that do not have any element of \ch\ as a minor. Then a graph $X$ is quasi-isometric to an element of \cc, \iff\ $X$ admits a map of bounded multiplicative distortion into an element of \cc. 
\end{observation}
\begin{proof}
The converse direction is trivial. For the forward direction, let $f: X \to G$ be a (M,A)-quasi-isometry with $G\in \cc$. For every $v\in V(G)$, we attach the centre of a star $S_v$ with $|V(X)|$ leaves to $v$, and subdivide each edge of $S_v$ into a path of length $\ceil{A}$. Let $G'$ be the resulting graph, and notice that $G'\in \cc$ because the elements of \ch\ are 2-connected. 

Define a map $f': V(X) \to V(G')$ as follows. For each $v\in V(G)$, let $f'$ map each vertex in $f^{-1}(v)$ to a distinct leaf of $S_v$; since $S_v$ has $|V(X)|$ leaves, this is possible. It is straightforward to check that $f'$ is a $((M+3\ceil{A}),0)$-quasi-isometry. 
\end{proof}

\subsection{Asymptotic dimension} \label{sec asdim}
We recall one of the standard definitions of asymptotic dimension $\asdim(X)$ of a metric space $(X,d)$  from e.g.\ \cite{JorLanGeo}. The reader can think of $X$ as being a graph endowed with its graph distance $d$. See  \cite{BelDraAsy} for a survey of this rich topic.
 \smallskip
 
Let \cu\ be a collection of subsets of $(X,d)$ ---usually a cover. For $s > 0$, we say that \cu\ is \defi{$s$-disjoint}, if
$d(U,U') := \inf\ \{d(x, x') \mid x \in U, x' \in U'\} \geq s$
whenever $U,U'\in \cu$ are distinct. More generally, we say that \cu\ is \defi{$(n+1,s)$-disjoint}, if $\cu = \bigcup_{i=1}^{n+1} \cu_i$ for subcollections $\cu_i$ each of which is $s$-disjoint. The indices $1,\ldots,n+1$ are the \defi{colours} of \cu. We define the \defi{asymptotic dimension} $\asdim(X)$ as the smallest $n$ \st\ \fe\ $s\in \R_+$ \ti\ an $(n+1,s)$-disjoint cover \cu\ of $X$ with $\sup_{U\in \cu} \diam(U)< \infty$. Here, the diameter $\diam(U)$ is measured \wrt\ the metric $d$ of $X$. We say that \cu\ is \defi{$D$-bounded} for any $D> \sup_{U\in \cu} \diam(U)$. 

The \defi{Assouad-Nagata dimension} $\andim(X)$ is defined by putting a stronger restriction on the diameters: we define $\andim(X)$  to be the smallest $n$ for which there is $c\in \R$ \st\ \fe\ $s\in \R_+$ \ti\ an $(n+1,s)$-disjoint and $cs$-bounded cover \cu\ of $X$.

\subsection{Fat minors} \label{sec FM}

\begin{Def}[Fat minors] \label{def fat}
Let $(X,d)$ be a metric space and $\Delta$ a graph. A $K$-\defi{fat} $\Delta$-\defi{minor} of $X$ is a collection of disjoint connected subspaces $B_v, v\in V(\Delta)$ of $X$ (called \defi{branch-sets}), 
and disjoint arcs $P_e,e\in E(\Delta)$ (called \defi{branch-paths}) such that:
\begin{enumerate}
\item \label{fm i} Each  $P_{uw}$ intersects $\bigcup_{v\in V(\Delta)} B_v$ exactly at its two endpoints, which lie in $B_u$ and $B_w$.

\item \label{fm ii} For any $Y,Z \in \{ B_v \mid v\in V(\Delta) \}\cup  \{P_e \mid e\in E(\Delta)\}$, we have $d(Y,Z)\geq K$, unless $Y=Z$, or $Y=P_{vw}$ and $Z=B_v$ for some $v,w\in V(\Delta)$, or vice versa.
\end{enumerate}
\end{Def}

\begin{definition}[Asymptotic minors] \label{def asym}
Let $\Delta$ be a finite graph. 
We say that $X$ has $\Delta$ as an \defi{asymptotic  minor}, and write $\Delta \prec^\infty X$,  if $X$ has $\Delta $ as a $K$-fat minor  for all $K>0$. \end{definition}

\begin{observation} \label{invariance}
The property of having a fixed finite graph $H$ as an \asmi\ is preserved under quasi-isometries between length spaces.
\end{observation}

In fact, \asmi s are preserved under the more general \defi{coarse embeddings} as defined in \cite{GroAsyInv,Ostrovskii}, but we will not use this fact here.

\comment{
\section{A Coarse Ramsey Theorem}

\begin{lemma} \label{ramsey}
(Coarse Ramsey Theorem)  Let $(X,d)$ be a metric space, and $U\subset X$ infinite. Then \fe\ $R\in \R_+$, there is either\\
A) an infinite $U'\subset U$ \st\ $d(u,v)>R$ \fe\ $u,v\in U'$, or \\
B) a ball of radius $R$ in $X$ containing all but finitely many elements of $U$.
\end{lemma}
\begin{proof}
If B) fails, then we can construct $U'$ greedily, by picking points $v_i$ of $U$ one by one, and removing the balls $B_{v_i}(R)$ from $U$.
\end{proof}
 
\begin{remark}

\end{remark}
}

\section{Manning's theorem --- quasi-trees} \label{quasitree}

Manning's theorem  \cite{Manning} says that a graph is quasi-isometric to a tree \iff\ it has the following \defi{$\delta$-\defi{Bottleneck Property}} for some $\delta\in \R_+$:

\begin{Def}
We say that a metric space $X$ has the $\delta$-\defi{Bottleneck Property},  if for all $x, y \in X$  there is a `midpoint' $m = m(x,y)$ with $d(x,m) = d(y,m) = d(x,y)/2$ such that any path from $x$ to $y$ is at distance less than $\delta$ from $m$.
\end{Def}

Our next result contains Manning's theorem, and at the same time proves the case $H=K_3$ of \Cnr{conj fat min}:

\begin{theorem} \label{triangle-free}
The following are equivalent for every graph $G$
\begin{enumerate}
\item \label{tf i} $G$ satisfies the $\delta$-Bottleneck Property for some $\delta>0$;
\item \label{tf ii} $G$ has no $K$-fat $K_3$-minor for some $K>0$; and
\item \label{tf iii} $G$ is quasi-isometric to a tree.
\end{enumerate}
\end{theorem}
\noindent {\bf Remark 1:} In this formulation our statement only makes sense for graphs $G$ of infinite diameter, but our proof yields explicit relations between the constants involved that make it meaningful for finite $G$ as well. \\
{\bf Remark 2:} Kr\"on \& M\"oller \cite{KroMolQua} provide an alternative characterisation of quasi-trees (for graphs $G$ of infinite diameter). Kerr \cite{KerTre} provides another proof of Manning's theorem, with multiplicative constant 1 for the quasi-isometry. Similar results had previously been obtained by Chepoi et al.\ \cite{CDNRV}. Our proof of \Tr{triangle-free} in fact also gives Kerr's theorem (with a different proof), as we will obtain a $(1,10K)$-quasi-isometry to a tree \eqref{fourteen}.  In other words, \ref{tf i}--\ref{tf iii} of \Tr{triangle-free} are equivalent to admitting an embedding into a tree with bounded additive distortion. 

\medskip
We will use the following terminology. Let $(X,d)$ be a metric space. We say that $C\subseteq X$ is \defi{$M$-near-connected} for $M\in \R_+$, if for every $x,y\in C$, there is a sequence $x=x_0,x_1, \ldots, x_k=y$ of points of $C$ such that $d(x_i,x_{i+1})\leq M$ \fe\ $i<k$. Such a sequence $P=\{x_i\}$ will be called an $M$-\defi{near path} from $x$ to $y$. 
An \defi{$M$-\ncm} is a maximal subset of $X$ that is $M$-near-connected. 

Our method below bears some similarity with \cite{CDNRV}, but our use of $M$-\ncm s instead of components in a layered partition is an essential difference. 

\begin{proof}[Proof of \Tr{triangle-free}]
\medskip
$\bullet$ \ref{tf i} $\to$ \ref{tf ii}:\\
Suppose that $G$ has a $2\delta$-fat $K_3$-minor $M$ with branch sets $B_1,B_2,B_3$ and paths $P_{12},P_{23}, P_{31}$. Let $x\in B_1, y\in B_2$ be the two endvertices of $P_{12}$. We may assume \obda\ that 
\labequ{dxB2}{d(x,B_2)=2\delta.}
Indeed, if \eqref{dxB2} fails, then let $x'$ be the last vertex along $P_{12}$ as we move from $x$ to $y$ \st\ $d(xP_{12}x',B_2)=2\delta$, where $xP_{12}x'$ denotes the subpath of $P_{12}$ from $x$ to $x'$. Such an $x'$ always exists since $d(x,B_2)\geq 2\delta$ and $y\in B_2$. Easily, $d(x',B_2)=2\delta$. We can modify $M$ by adding $xP_{12}x'$ to $B_1$ and replacing $P_{12}$ by $x' P_{12}y$, to obtain a new $K_3$-minor $M'$ of $G$ which satisfies \eqref{dxB2} by the choice of $x'$. It is straightforward to check that $M'$ is still $2\delta$-fat. 

Next, we claim that we may assume each $B_i$ to be a path with end-vertices in $\bigcup_{j\neq i} P_{ij}$. Indeed, each $B_i$ contains such a path $P_i$ by the definitions, and discarding $B_i\sm P_i$ cannot reduce the fatness of $M$. This means that the vertices and edges of $G$ in the branch sets and paths of $M$ form a cycle $C$ of $G$.

Let $\gamma$ be a $x$--$B_2$~geodesic,  let $m$ be its midpoint, and let $b\in B_2$ be its final vertex. Thus $d(m,B_2)=\delta$ by  \eqref{dxB2}. Let $C_1$ be the subpath of $C$ from $x$ to $b$ containing $P_{12}$, and let $C_2$ be the other subpath of $C$ from $x$ to $b$.

As we are assuming \ref{tf i}, there is a point $z_1$ in $C_1$ \st\ $d(z_1,m)<\delta$, and a point $z_2$ in $C_2$ \st\ $d(z_2,m)<\delta$. By the triangle inequality we thus have
\labequ{zi}{d(z_1,z_2)<2\delta.}
Likewise, since $d(m,B_2)=\delta$, we also obtain 
\labtequ{ziB2}{$d(z_i,B_2)<2\delta$ and $d(z_i,B_2)>2\delta-\delta=\delta$ for $i=1,2$.}
In particular, neither of $z_1,z_2$ lies in $B_2$. Since all of  $B_1,B_3$ and $P_{13}$ are at distance at least $2\delta$ from $B_2$, we must then have $z_1\in P_{12}$ and $z_2\in P_{23}$. But combined with \eqref{zi} this implies $d(P_{12},P_{23})<2\delta$, contradicting that $M$ is $2\delta$-fat.
\bigskip

$\bullet$ \ref{tf ii} $\to$ \ref{tf iii}:

Fix a `root' $o\in V(G)$, and let $S_n:= \{ v\in V(G) \mid d(v,o)=n \}$ be the sphere at distance $n\in \N$. 

Let $\cc_n$ denote the set of  $5K$-near-components of $S_n$, where we endow $S_n$ with the metric $d=d_G$ inherited from $G$.\\

{\bf Claim 1}: Each $C\in \cc_n$ has diameter at most $D=D(K):= 10K$ \fe\ $n\in \N$.\\

If $n\leq D/2$ this is obvious, so assume $n> D/2$ from now on. \mymargin{This paragraph has been rewritten.}
Suppose, to the contrary, there is $C\in \cc_n$ and $x_1,x_2\in C$ with $d(x_1,x_2)>D$. Let $P$ be a \pth{x_1}{x_2}\ in \g such that %
$d(P,S_n)\leq 5K/2$, 
which exists by the definition of a $5K$-near-component and the fact that $C \subseteq S_n$. 
Let $\pi_{x_i o}$  be a geodesic from $x_i$ to $o$ for $i=1,2$, and let $B_i$ be the component of $x_i$ in $(\pi_{x_i o} \cup P) \cap B(x_i,7K/2)$, 
 where $B(x,R)$ denotes the ball of radius $R$ with center $x$ in $G$. Let $P_{x_i o}$ be the subarc of $\pi_{x_i o}$ of length $K$ starting at the last point of $B_i$ (which point is an interior point of an edge unless $K$ is an even integer). 
Let $P_{x_1 x_2}$ be a subpath of $P$ joining $B_1$ to $B_2$. Note that $d(P_{x_i o}, S_n)=  7K/2$, and therefore 
$d(P_{x_1 x_2}, P_{x_i o})\geq K$ by the triangle inequality because $d(P,S_n)\leq 5K/2$. 
Let $B_o= B(o,n-9K/2)$, and note that this implies $d(B_o, P)\geq 2K$ and $d(B_o, B_i)\geq K$ for similar reasons.  We claim that $d(P_{x_1 o}, P_{x_2 o})\geq K$. Indeed, if not, then let $Q$ be a \pth{P_{x_1 o}}{P_{x_2 o}} of length less than $K$. Then by concatenating subpaths of the five paths $(\pi_{x_1 o} \cap B_1), P_{x_1 o}, Q, P_{x_2 o}, (\pi_{x_2 o} \cap B_2) $ we obtain an $\pth{x}{y}$ of length less than $10K$, contradicting our assumption. 

It is now  straightforward to check, using the above inequalities, that the above paths form a $K$-fat $K_3$-minor of \g with branch sets $B_1,B_2,B_o$. This contradicts our assumption thus proving Claim~1.

\bigskip

Define a tree $T$ as follows. The vertex set  $V(T):= \bigcup_n  \bigcup \cc_n$ is the set of all $5K$-near-components of all $S_n$.
We put an edge between $C\in \cc_i$ and $C' \in \cc_j$ in $T$ whenever there is an edge of $G$ between $C$ and $C'$. We can think of $\cc_0$, which is just a singleton $\{\{o\}\}$, as the root of $T$. Notice that there is no edge of $T$ between distinct elements of $\cc_i$ by the definition of $M$-near-component. The fact that $T$ is a tree thus follows from\\

{\bf Claim 2}: \Fe\ $\nin$ and every $C\in \cc_{n+1}$, there is exactly one $C'\in \cc_n$ such that $CC' \in E(T)$. \\

Indeed, suppose there are $C'\neq C''\in \cc_n$ connected to $C$ by an edge. In this case, there is a path $A$ in \g \st\  the endvertices $a,z$ of $A$ lie in distinct elements of $\cc_n$ (which elements are not necessarily $C'$ and $C''$), and the interior of $A$ is disjoint from $\bigcup_{i\leq n} \bigcup \cc_i$. To see this, let $e$ be an edge of \g joining $C'$ to $C$, let $f$ be an edge of \g joining $C''$ to $C$, and let $A'$ be a path in \g joining the endvertices of $e$ and $f$ lying in $C$ \st\ $A'$ is at distance at most $5K$ from $C$; such an $A'$ exists since $C$ is $5K$-near-connected. It is easy to see that the concatenation $e A' f$ contains a subpath $A$ with the desired properties.

Note that 
\labtequ{az}{$d(a,z)>5K$}
since $a,z$ lie in distinct $5K$-near-components. Let $\pi_{a o}$ and $\pi_{z o}$ be a geodesic from $a$ to $o$  and a geodesic from $z$ to $o$ respectively.  Let $Z$ be an \pth{a}{z} contained in their union $\pi_{a o} \cup \pi_{z o}$.

We construct a $K$-fat $K_3$-minor in \G\ as follows.
Let $B_a$ be the initial subpath of $\pi_{a o}$ of length $2K$ and let $B_z$ be the initial subpath of $\pi_{z o}$ of length $2K$. Notice that $d(B_a,B_z) > K$ by \eqref{az}. Let $P_{az}= Z \sm \{B_a \cup B_z\}$. Let $p$ be the last point along $A$, as we move from $a$ to $z$, \st\ $d(p,B_a)\leq K$ holds, and let $q$ be the first point along $A$ after $p$ \st\ $d(q,B_z)\leq K$ holds. Let $B_w$ be the subpath of $A$ between $p$ and $q$. Easily, $d(B_w,B_a), d(B_w,B_z)\geq K$ hold. Let $P_{aw}$ be a shortest $p$--$B_a$~path in $G$, and let $P_{wz}$ be a shortest $q$--$B_z$~path in $G$. Note that $|P_{aw}|=|P_{wz}|=K$. 

We claim that $d(P_{aw},P_{wz}) \geq K$. To see this, suppose there is a \pth{P_{aw}}{P_{wz}} $Q$ with $|Q|<K$. Let $Q_a$ be the subpath of $\pi_{a o}$ from $a$ to $P_{aw}$, and notice that $|Q_a|\leq K$ because $P_{aw}$ starts at $S_{n+1}$ and has length $K$, and therefore cannot meet a point of $\pi_{a o}$ at distance $K$ or more  from $a\in S_n$. We define $Q_z$ similarly, as a subpath of $\pi_{z o}$. \mymargin{This paragraph has been rewritten.} Thus by concatenating subpaths of the five paths $Q_a, P_{aw}, Q, P_{wz}, Q_z$ we obtain an \pth{a}{z}\ of length less than $5K$, contradicting \eqref{az}. This proves $d(P_{aw},P_{wz}) \geq K$. Similarly, we have $d(P_{aw},B_z)\geq K$ and $d(P_{wz},B_a)\geq  K$. 


Finally, note that both $P_{aw},P_{wz}$ start at distance at least $n+1$ from $o$, 
while all of $P_{az}$ is within distance $n-2K$ from $o$. Combining the last two inequalities we deduce $d(P_{aw}, P_{az}), d(P_{wz}, P_{az})> K$. Similarly, we have $d(B_w, P_{az})> K$. 

Thus these sets form a   $K$-fat $K_3$-minor. 
\comment{Old wrong version: Pick a point $p$ on $A$ \st\ 
\labtequ{apz}{$d(a,p) = d(p,z)= \frac{d(a,z)}{2}>2K$.}
Let $Z$ be a \pth{a}{z} contained in their union $\pi_{a o} \cup \pi_{z o}$. By choosing three appropriate branch sets around $a, p$ and $z$ we can find a $K$-fat $K_3$-minor in $A \cup Z$ as follows. Let $B_a$ be the initial subpath of $\pi_{a o}$ of length $K$ and let $B_z$ be the initial subpath of $\pi_{z o}$ of length $K$. Let $a'$ be the last vertex of $A$, as we move from $a$ towards $z$, with $d(B_a,a')= K$ ---which exists by \eqref{apz}--- and similarly, let $z'$ be the first vertex of $A$ with $d(B_z,z')= K$. Then let $B_p$ be the subpath of $A$ from $a'$ to $z'$, and note that, by the choice of $a',z'$, we have 
\labtequ{Bp}{$d(B_p,B_a), d(B_p,B_z)\geq K$.}
Finally, let $P_{ap}$ and $P_{pz}$ be the two subpaths of $A \sm B_p$, and let $P_{az}= Z \sm \{B_a \cup B_z\}$. It is straightforward to check that these sets form a   $K$-fat $K_3$-minor using \eqref{az}, \eqref{Bp}, and the fact that $P_{az}$ is within distance $n-K$ from $o$ while $d(A,o)=n$ by the definitions.} 
This contradicts our assumption \ref{tf ii}, hence proving Claim~2.

\

Notice that  $T$ is obtained from $G$ by contracting disjoint sets (the above $5K$-near-components) into points. Since these sets have diameter bounded by $10K$ by Claim~1, the map $f: V(G) \to V(T)$ mapping each vertex to the $5K$-near-component it belongs to defines a quasi-isometry from $G$ to $T$. More precisely, we claim that 
\labtequ{fourteen}{$f$ is a $(1,10K)$-quasi-isometry from $G$ to $T$.}
Indeed, given $x,y\in V(G)$, let $P$ be an \pth{x}{y}\ in \g with length $d_G(x,y)$, and let $Q$ be the unique \pth{f(x)}{f(y)}\ in $T$. It is easy to see that $d_G(x,y)\geq d_T(f(x),f(y))$, because $f(P)$ is connected, hence it must contain all edges of $Q$. 

For the converse, let $Z$ be the vertex of $Q$ that is closest to the root $f(o)$ of $T$. Let $Q_1$ be the (possibly trivial) subpath of $Q$  from $f(x)$ to $Z$, and $Q_2$ the subpath  from $Z$  to $f(y)$. Let $\Pi_x$ (respectively,  $\Pi_y$) be the $x$--$o$ (resp.\ $y$--$o$) geodesic in \G. Let $z_x$ be the unique vertex of $\Pi_x \cap Z$, and $z_y$ the unique vertex of $\Pi_y \cap Z$. Let $S$ be a shortest \pth{z_x}{z_y}\ in \G, and let \defi{$||S||$} denote its length. Then $||S||\leq 10K$ by Claim 1. We can form an  \pth{x}{y}\ by concatenating $x \Pi_x z_x$ with $S$ and then $z_y \Pi_y y$. Notice that $||x \Pi_x z_x|| = ||Q_1|| $, and $||z_y \Pi_y y|| = ||Q_2||$ (where $x \Pi_x z_x$ stands for the $x$--$z_x$~subpath of $\Pi_x$). It follows that\\  
 $d(x,y)\leq 10K + ||Q|| = 10K + d_T(f(x),f(y))$ as desired.
\medskip

$\bullet$ The implication \ref{tf iii} $\to$ \ref{tf i} is obvious.
\end{proof}

\comment{
The \defi{$k$-fan} $\fan_k$ is the graph obtained from a path $P$ of $k$ vertices by adding a new vertex and joining it to \mymargin{new paragraph} each vertex of $P$ with an edge. By a straightforward modification of the above proof we will obtain the following fact 
\begin{corollary} \label{corol fan}
Let $G$ be a graph with no $K$-fat $\fan_k$-minor, and $M\in \N_*$. Then there is $D=D(K,k,M)$ such that each $MK$-near-component of any `sphere' $S_n, n\in \N$ (defined as above) has diameter at most $D$.
\end{corollary}
\begin{proof}
We argue as in the proof of Claim~1, except that we now pick a sequence $x_1,\ldots, x_k$ of points along the path $P$ with pairwise distance at least $2(M+2)K$, we let $B_o= B(o,n-(M+2)K)$, and let $B_i$ be the component of $x_i$ in\\ $(\pi_{x_i o} \cup P) \cap B(x_i,(M+1)K)$.
\end{proof}
}

\section{\Cnr{conj fat min} for stars} \label{sec stars}

In this section we establish the special case of \Cnr{conj fat min} where $H$ is a \defi{star}, i.e.\ $H=K_{1,m}$:

\begin{theorem} \label{thm star}
Let $X$ be a length space, and $m\in \N_{>0}$. Then $X$ has no asymptotic $K_{1,m}$ minor \iff\ $X$ is quasi-isometric to a graph with no $K_{1,m}$ minor.
\end{theorem}
\begin{proof}
By \Or{qi graph} we may assume that $X$ is a graph. The backward implication is immediate from \Or{invariance}. We will prove the forward implication: we assume that $X$ does not have a $K$-\defi{fat} $K_{1,m}$ minor, and will construct a $K_{1,m}$-minor-free graph quasi-isometric to $X$. 

\medskip
We follow the layered partition technique of the proof of \Tr{triangle-free}, however, instead of the metric \defi{spheres} $S_n$ we will now decompose $X$ into its metric \defi{annuli} of width $M:=4K+1$. More precisely, fix a `root' $o\in V(X)$, and let $A_n:= \{ v\in V(X) \mid d(v,o)\in [nM, (n+1)M) \}, n\in \N$ be the $n$th annulus.

Let $\cc_n$ denote the set of  $K$-near-components of $A_n$, where we endow $A_n$ with the metric $d=d_G$ inherited from $G$. An element of $\cc_n, \nin$, is called a \defi{box}. Similarly to {Claim 1} from the proof of \Tr{triangle-free}, we have\\

{\bf Claim a}: Each box $C\in \cc_n$ has diameter less than $D:= c K^2$ \fe\ $n\in \N$, where $c=c_m$ is a universal constant.\\

Indeed, this follows by combining the argument of {Claim 1} with a pigeonhole argument: given $C\in \cc_n$, and $x,y\in C$ with $d(x,y)= \diam(C)$, let $P$ be an \pth{x}{y}\ contained in the ball $B$ of radius $K/2$ around $A_n$, which exists by the definition of a $K$-near-component. If $d(x,y)$ is large enough, then we can pick points $x_0,\ldots, x_{m(M+K)}$ along $P$ with pairwise distance at least $3K$. Since $B$ has bounded width, i.e.\ $B\subset \{v\in V(X) \mid d(v,o)\in [nM-K/2, (n+1)M)+K/2 \}$, the pigeonhole principle implies that some subset $\{x'_1, \ldots x'_m\}$ of the $x_i$ have equal distance to $o$. Letting $\pi_i$ be a geodesic from each $x'_i$ to $o$, we can easily find a $K$-fat $K_{1,m}$ minor of $X$ inside $\bigcup \pi_i$, letting the singletons $\{x'_i\}$ be the branch sets corresponding to the leaves of  $K_{1,m}$. This contradiction bounds $d(x,y)$ and therefore proves Claim~a. (The resulting constant $c_m$ is at most $15m$.)
\medskip

Call a box $C\in \cc_n$ \defi{short}, if it send no edge to $\cc_{n+1}$, and call it \defi{tall} otherwise. As in the proof of \Tr{triangle-free}, we are going to form a graph \g by contracting each box into a vertex. Our hope that \g will be $K_{1,m}$-minor-free will not quite come true, because a $K_{1,m}$ minor in \g that involves some short components will not give rise to a fat $K_{1,m}$ minor in $X$. Therefore,  we want to suppress the short boxes, and we do so as follows. Call two tall boxes $C_1,C_2 \in \cc_n$ \defi{siblings}, if they send an edge to a common short box (of $\cc_{n+1}$). Let $\sim$ denote the transitive reflexive closure of the sibling relation. For each $\sim$-equivalence class $[C]$, define a \defi{super-box} $C'$ comprising the union of all elements of $[C]$ as well as all short boxes in $\cc_{n+1}$ sending an edge to an element of $[C]$. Finally, let $\cc'_n$ denote the set of super-boxes $C'$ with $C\in \cc_n$. Notice that distinct super-boxes are disjoint, and therefore $\bigcup_{\nin} \cc'_n$ is a partition of $V(X)$. Moreover, for each $C' \in \cc'_n$, we have $C' \subseteq \cc_n \cup \cc_{n+1}$ by definition, and therefore 
\labtequ{diam C}{$\diam(C') < 2D.$}
Indeed, we can repeat the proof of Claim a with $M$ replaced by $2M$. 

Each super-box $C'\in \cc_n$ contains at least one tall box $C$, which by definition has a non-empty intersection $C_M$ with the \defi{middle-layer} of $A_n$, i.e.\ the set 
$A^{1/2}_n:= \{ v\in V(X) \mid d(v,o)=nM + 2K+1 \}$. We call the vertices in $C' \cap A^{1/2}_n$ the \defi{mid-points} of $C'$ ---we could have alternatively defined a box to be \defi{tall} if it has at least one mid-point. Mid-points are helpful because they are far from any other super-box:
\labtequ{midp}{Let $C'_1\neq C'_2$ be two super-boxes, and $x_i$ a mid-point of $C'_i$. Then $d(x_1,C'_2)>2K$, and  $d(x_1,x_2)>4K$.}
Indeed,  every path from $x_i$ to the complement of $C'_i$ has length greater than $2K$ ---because the width of $A_n$ is $4K+1$--- from which both inequalities easily follow. 

\medskip
We now define the graph $G:= X/ \bigcup_\nin \cc'_n$ by contracting each super-box of $X$ into a vertex. In other words, we let $V(G):= \bigcup_\nin \cc'_n$, and we put an edge between $C'_1,C'_2$ in \g whenever there is an edge between these vertex-sets in $X$. 

As in the proof of \Tr{triangle-free}, \eqref{diam C} implies that \g is $(2D,2D)$-quasi-isometric to $X$. Therefore, it only remains to prove that
\labtequ{nominor}{\g does not have a $K_{1,m}$ minor.}
Suppose to the contrary that such a minor $\Delta$ exists, and let $B$ be the branch set corresponding to the degree-$m$ vertex and $v_1,\ldots v_m$ the branch set corresponding to the leaves of $K_{1,m}$. Easily, we may assume that each $v_i$ is a vertex of \G. We will turn $\Delta$ into a $K$-fat $K_{1,m}$ minor of $X$, a contradiction to our assumption.

Recall that each $v_i$ is a super-box of $X$, and so we can pick a mid-point $x_i\in v_i$. Let $\bar{B}$ be a connected subgraph of the ball of radius $K/2$ around $\bigcup B$ in $X$ containing $\bigcup B$, which exists because $\bigcup B$ is $K$-near-connected being a union of adjacent $K$-near-components. Notice that \eqref{midp} implies that $d(x_i,\bar{B})>3K/2$, and  $d(x_i,x_j)>4K$ for $i\neq j$.

Let $\pi_i, 1 \leq i \leq m$, be geodesic from $x_i$ to $\bar{B}$, and let $P_i$ be the initial subpath of $\pi_i$ of length $K$. We can now form a $K$-fat $K_{1,m}$ minor of $X$ as follows. We let the singletons $\{x_i\}$ be the  branch sets corresponding to the $m$ leaves of $K_{1,m}$. We let $P_i$ be the branch paths corresponding to the edges, and we let $\bar{B} \cup \bigcup_{1\leq i \leq m} \pi_i \sm Q_i$ be the branch set corresponding to the centre.  These sets form the desired $K$-fat $K_{1,m}$ minor of $X$. This contradiction proves \eqref{nominor}, completing our proof.
\end{proof}

\noindent  {\bf Remark:} In the special case $m=3$ of \Tr{thm star}, the resulting graph \g becomes a path or a cycle. In this case, after subdividing each edge of \g $M$ times, the quasi-isometry from $X$ to $G$ ---obtained by mapping each vertex to its super-box--- becomes a map of bounded additive distortion, i.e.\ a $(1,K')$-quasi-isometry. This can be proved similarly to \eqref{fourteen}. We do not see how to extend this to the case $m>3$.

\section{More on \Cnr{conj fat min}} \label{sec homeo}

Despite our \Cnr{conj fat min}, we remark that the correspondence between minor-closed families of graphs and quasi-isometry classes is not 1--1. \Tr{thm subdiv} below makes this point clear, but it also provides a tool for studying \Cnr{conj fat min}. Here we think of graphs as topological spaces, by considering the corresponding  1-complex.

\begin{theorem} \label{thm subdiv}
Let $\Gamma,\Delta$ be finite graphs that are homeomorphic as topological spaces, and let $X$ be a length space. Then  \asm{\Gamma}{X} \iff\ \asm{\Delta}{X}. 
\end{theorem}

Let  $\Theta$ be the graph obtained from $K_4$ by removing one edge (or from $K_{2,3}$ by contracting one). Recall that outerplanar graphs can be characterised as $\mathrm{Forb}(K_4, K_{2,3})$. Combining this with  \Tr{thm subdiv}, and a result of Fujiwara \& Papasoglou \cite{FujPapCoa} stating that a graph is quasi-isometric to a cactus \iff\ it has no asymptotic $\Theta$ minor, we obtain

\begin{corollary} \label{cor OP}
The following are equivalent for a length space $X$:
\begin{enumerate}
         \item \label{OP ii} $X$ is quasi-isometric to a cactus; 
         \item \label{OP i} $X$ is quasi-isometric to an outerplanar graph;
      \item \label{OP iii} $X$ has no asymptotic $K_4$ or  $K_{2,3}$ minor, and
      \item \label{OP iv} \ti\ a subdivision $\Theta'$ of $\Theta$ \st\ $X$ has no asymptotic $\Theta'$ minor.
\end{enumerate}
\end{corollary}
Here, we say that a graph $\Delta'$ is a \defi{subdivision} of a graph $\Delta$, if $\Delta'$ is obtained by replacing some of the edges of $\Delta$ by independent paths with the same end-vertices.

\begin{proof}
The implication \ref{OP ii} $\to$ \ref{OP i} is obvious because every cactus is outerplanar.

The implication \ref{OP i} $\to$ \ref{OP iii} is straightforward since quasi-isometric spaces have the same asymptotic minors and an outerplanar graph has no $K_4$ or  $K_{2,3}$ minor, let alone an asymptotic one. 

The implication \ref{OP iii} $\to$ \ref{OP iv} is obvious because $K_{2,3}$ is a subdivision of $\Theta$.

To prove \ref{OP iv} $\to$ \ref{OP ii}, notice first that $X$ is quasi-isometric to a graph \G, which by the above remark has no asymptotic $\Theta'$ minor. Then \g has no asymptotic $\Theta$ minor by \Tr{thm subdiv}. Finally, \g is quasi-isometric to a cactus by the aforementioned result of Fujiwara \& Papasoglu \cite{FujPapCoa}, hence so is $X$.
\end{proof}

Note that two graphs are homeomorphic (as topological spaces) \iff\ they are subdivisions of the same graph (which can be obtained by suppressing all vertices of degree 2). Therefore \Tr{thm subdiv} follows by applying the following lemma twice:

\begin{lemma} \label{lem subdiv}
Let $\Delta'$ be a subdivision of a finite graph $\Delta$, and let $X$ be a length space. Then \asm{\Delta}{X} \iff\ \asm{\Delta'}{X}.
\end{lemma}
\begin{proof}
For simplicity, let us assume that $X$ is a geodesic metric space, it is straightforward to adapt the following arguments to more general length spaces.

By the definitions, \asm{\Delta'}{X} implies \asm{\Delta}{X} because $\Delta \prec \Delta'$.
To show the converse, it suffices to consider the case where $\Delta'$ is obtained from $\Delta$ by subdividing an edge $e\in E(\Delta)$ once, since we can then repeatedly subdivide edges to obtain any subdivision  $\Delta'$ of $\Delta$.

So given $e=xy\in E(\Delta)$, a geodesic space $X$ with \asm{\Delta}{X}, and $K\in \N$, let $M=(\{V_x, x\in V(\Delta)\}, \{P_f, f\in E(\Delta)\})$ be a $3K$-fat $\Delta$-minor in $X$. We construct a $K$-fat $\Delta'$-minor $M'$ in $X$ as follows. For every $x\in V(\Delta)$, we just keep $V'_x:= V_x$ as a branch set. We need to provide the branch set $V'_o$ of the vertex $o$ of $\Delta'$ arising from the subdivision of $e$. For this, let $p$ be the last point along $P_e$, as we move from $V_x$ to $V_y$, \st\ $d(p,V_x)\leq K$ holds, and let $q$ be the first point along $P_e$ after $p$ \st\ $d(q,V_y)\leq K$ holds. Let $V'_o$ be the subpath of $P_e$ between $p$ and $q$. Easily, $d(V'_o,V_z)\geq K$ holds \fe\ $z\in V(\Delta)$. Let $P'_{xo}$ be a shortest $p$--$V_x$~path in $X$, and let $P'_{yo}$ be a shortest $q$--$V_y$~path in $X$, and notice that $|P'_{xo}|=|P'_{yo}|=K$. For every other edge $f$ of $\Delta'$ we have $f\in E(\Delta)$ and we just let $P'_f:=P_f$ to complete the definition of $M'=(\{V'_x, x\in V(\Delta')\}, \{P'_f, f\in E(\Delta')\})$. We have $d(P'_{xo},P'_f)>K$ \fe\ $f\in E(\Delta)$ because $d(p,P_f)\geq 3K$ and $P'_{xo}\subseteq B_p(K)$. Moreover, $d(P'_{xo},P'_{yo})\geq K$ because $P'_{xo},P'_{yo}$ are within distance $K$ from $V_x,V_y$, respectively, and $d(V_x,V_y)\geq 3K$. 

Thus $M'$ is a $K$-fat $\Delta'$-minor in $X$, establishing \asm{\Delta'}{X}. 

\end{proof}

The converse of \Tr{thm subdiv} is false in general: for example, we can let $\Gamma$ be $K_2$, and $\Delta$ be its complement. But if we insist that $\Gamma,\Delta$ are both connected, and allow $X$ to be disconnected, i.e.\ we allow for infinite distances between pairs of points of $X$, then the converse becomes true and not hard to prove\footnote{For this, given non-homeomorphic $\Gamma_1,\Gamma_2$, let $X_i, i=1,2$ be the disjoint union of \oo\ copies of $\Gamma_i$, with each edge of the $n$th copy subdivided $n$ times. Then  $\asm{\Gamma_i}{X_i}$, but $\nasm{\Gamma_i}{X_{3-i}}$.}. If we insist that $X$ be connected however, then the converse fails: if $\Gamma$ is a cycle, and $\Delta$ the union of two cycles with exactly one common vertex, then $\Gamma,\Delta$ are asymptotically equivalent though non-homeomorphic. 
Still, we can obtain a converse to \Tr{thm subdiv} (for connected $X$) by restricting to 2-connected graphs. We leave the details to the interested reader.


\section{Coarse K\"onig Theorem} \label{sec Konig}
K\"onig's classical theorem states that every infinite connected graph contains a countably infinite star or a ray as a subgraph \cite[Proposition 8.2.1.]{DiestelBook05}. It is a cornerstone of infinite graph theory as it facilitates compactness arguments. In this section we prove a metric analogue. 
\smallskip

Let $P_n$ denote the path (in the graph theoretic sense) of length $n$. Let $\bigvee P_n$ be the graph obtained from the disjoint union of \seq{P}\ by identifying the first vertex of each $P_n$. A \defi{ray} is a 1-way infinite path.

\begin{theorem} \label{thm Konig}
(Coarse K\"onig Theorem) Let $X$ be a (connected) length space of infinite diameter. Then $X$ has either the ray or $\bigvee P_n$ as an asymptotic minor.
\end{theorem}

We prepare the proof of this with a couple of lemmas of independent interest. A ray $P$ in a graph \g is called \defi{divergent}, if $P\cap S$ is finite \fe\ subset $S$ of $V(G)$ with bounded diameter. 

\begin{lemma} \label{lem ray}
Let $G$ be a graph that has a divergent ray. Then $G$ has  the ray as an asymptotic minor.
\end{lemma}
\begin{proof}
Let $P$ be a divergent ray of $G$, and let $o$ be its starting vertex. Given $R\in \N$, let $r(n)$ denote the last vertex $v$ of $P$ with $d(o,v)\leq n$. 
We form an $R$-fat ray minor of \g as follows. Let $B_0:= \{o\}$. 
Having defined $B_{i-1}$, we proceed to define $B_i, i\geq 1$ inductively by letting $B_i$ be the subpath of $P$ from the last vertex of $B_{i-1}$ to $r({R+s_n})$, where $s_n:= \max \{d(o,v) \mid v\in  \bigcup_{j<i }V(B_j) \} $. This choice guarantees that $d(B_i,B_j)\geq R$ holds whenever $|i-j|\geq 2$ (\fig{figFatRay}).

\begin{figure} 
\begin{center}
\begin{overpic}[width=1\linewidth]{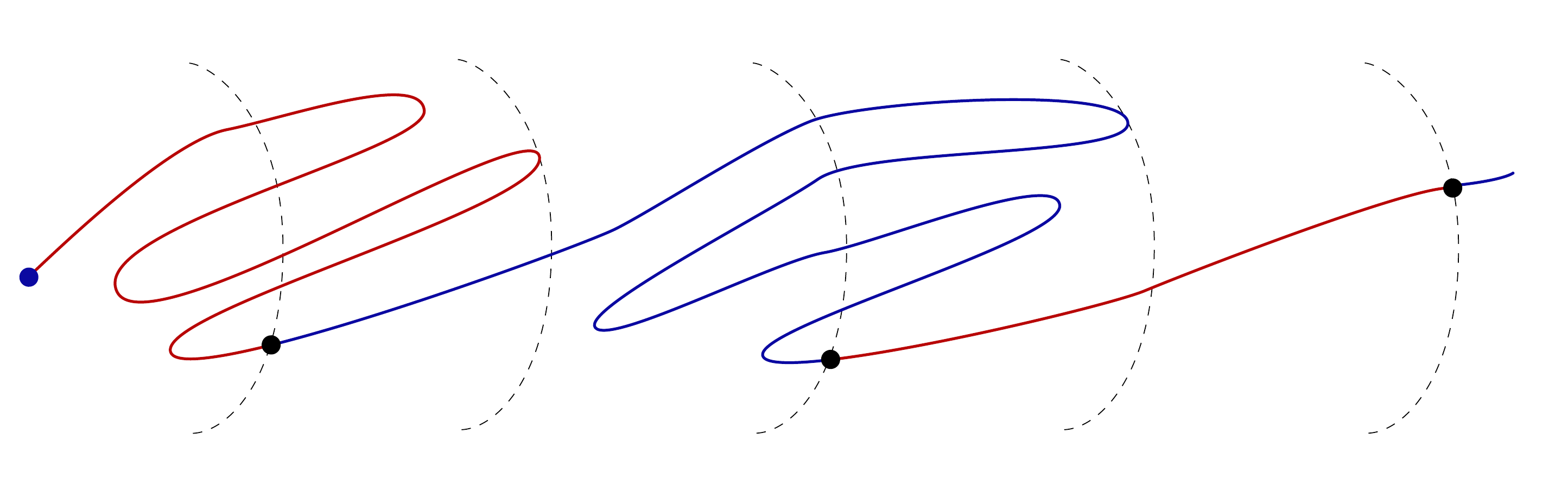} 
\put(0,11){$o$}
\put(11,28){$R$}
\put(17.2,6.4){$r(R)$}
\put(28,28){$s_2$}
\put(53,5.6){$r(R+s_2)$}
\put(94,16.4){$r(R+s_3)$}
\put(44,28){$R+s_2$}
\put(66,28){$s_3$}
\put(83,28){$R+s_4$}
\end{overpic}
\end{center}
\caption{Turning a diverging ray into a fat ray minor. Points meeting the same dashed line have the same distance from $o$,  indicated at the top. Branch sets are depicted in blue, if colour is shown, and branch paths in red.} \label{figFatRay}
\end{figure}

Thus we can let the $B_i$ with even $i$ be the branch sets, and the $B_i$ with odd $i$ the branch paths, of a ray-minor of \G, which minor is $R$-fat by the last remark.
\end{proof}

By the Infinite Ramsey theorem, every infinite graph contains either an infinite clique or an infinite independent set. We will use the following corollary of this.
\begin{lemma} \label{Ramsey}
Let $G_1 \supseteq G_2 \supseteq G_3 \ldots$ be graphs with a common infinite vertex set $V$. Then either some $G_i$ contains an infinite independent set, or there are infinite subsets $S_1 \supseteq S_2 \supseteq S_3 \ldots$ of $V$ \st\ $G_i[S_i]$ is complete \fe\ $i$.
\end{lemma}
\begin{proof}
Apply the Infinite Ramsey theorem to $G_1$. Either it returns an independent set and we are done, or it returns a clique $G_1[S_1]$. Apply the Infinite Ramsey theorem  to the subgraph  $G_2[S_1]$ induced on $G_2$ by the vertices of that clique, and proceed recursively. In each step $i$, either we find an infinite independent set in $G_i$, or a clique $G_i[S_i]$ where $S_i \subseteq S_{i-1} \ldots \subseteq S_1$. 
\end{proof}

We can now prove the main result of this section.
\begin{proof}[Proof of \Tr{thm Konig}]
For convenience we will assume $X$ to be a graph, which we can because every length space is quasi-isometric to a graph (\Or{qi graph}). 

Since $X$ has infinite diameter, there is a sequence $\seq{P}$ of geodesics $P_n:[0,n]\to X$ parametrised by arc length with $P_n(0)=o\in X$ a fixed point.
Call this set of geodesic paths $S$.
Fix $R>0, R\in \mathbb N$. We will show that $X$ contains either an $R$-fat $\bigvee P_n$
or an $R$-fat ray.

We say that $\{P_i,P_j\}$ is a \defi{$(R,n)$-bad pair}, if there are  vertices $v_i\in V(P_i), v_j\in V(P_j)$ satisfying $d(v_i,v_j)\leq R$ and $d(v_i,o), d(v_j,o) > n$.

We define a sequence of auxiliary graphs on $S$ as follows. Let $G_n=G_n(R)$ be the graph on $S$ whose edges are the $(R,n)$-bad pairs. Apply \Lr{Ramsey} to this sequence \seq{G}. There are two cases depending on the outcome that \Lr{Ramsey} returns. 

If it returns some $G_n$ with an independent set $S_n$, then we can find a $\bigvee P_n$ $R$-fat minor in $\bigcup S_n$ as follows. Choose a sequence $\{P'_k\}_{k\in \N}$ of elements of $S_n$
such that the length of $P'_k$ is greater than $2kR+n$. We set 
$$B_{v_0}=\bigcup_{k\in \N} P'_k([0,n]),$$ 
which will be the branch set corresponding to the infinite-degree vertex of $\bigvee P_n$. For the other branch sets, we decompose each $P'_k \sm P'_k([0,n])$ into $2k$ intervals of length at least $R$, and let every other such interval be a branch path and every other one a branch set.
This defines an $\bigvee P_n$ minor in $X$, which is $R$-fat because no pair $\{P'_k,P'_j\}$ is $(R,n)$-bad, and each $P'_k$ is a geodesic. 

\medskip
Otherwise, our application of \Lr{Ramsey} returns a sequence $S_1 \supseteq S_2 \supseteq S_3 \ldots$ with each $G_i[S_i]$ complete. In this case, we will find a divergent ray $D$ in $X$, and therefore an asymptotic ray minor by \Lr{lem ray}. We construct $D$ recursively as follows. Let $D_0$ be the trivial path $\{o\}$, and pick an element  $P^1$ of $S_1$. Then for $i=1,2, \ldots$, assume we have defined a connected subgraph $D_{i-1}$ of $X$ and some $P^{i}\in S_i$ such that $D_{i-1} \cap P^{i}$ contains a vertex $h_i$ with $d(o,h_i)\geq i-1-R$. Pick $P^{i+1}\in S_{i+1}$, and note that since $G_{i}[S_{i}]$ is complete, and $P^{i+1}\in S_{i+1} \subseteq S_i$, the pair $\{P^{i}, P^{i+1}\}$ is $(R,i)$-bad. This means that there is a \pth{P^{i}}{P^{i+1}}\ $X_i$ of length at most $R$, and at distance at least $i-R$ from $o$. Note that $\lim_{j\to \infty} d(o, X_j)= \infty$ as $R$ is fixed.  Let $Z_i$ be the subpath of $P^{i}$ from $X_i$ to $h_{i}$, and let $D_{i}:=D_{i-1} \cup X_i \cup Z_i$. Letting $h_{i+1}:= P^{i+1} \cap X_i$,  our inductive hypothesis is preserved. Note that $\lim_{j\to \infty} d(o, Z_j)$ is infinite as well. 

Let $D':= \bigcup_{i\in \N} D_i$. Clearly, $D'$ is connected since each $D_i$ is. We claim that $D'$ is \lf. 
Indeed, notice that each $D_i$ is finite, and it meets at most finitely many of the  $X_j \cup Z_j, j\in \N$, because $\lim_{j\to \infty} d(o, X_j \cup Z_j)= \infty$.

Let $D$ be a ray in $D'$, which exists by K\"onig's theorem mentioned at the beginning of this section. Then $D$ is divergent by the last argument.

\end{proof}

	\comment{For convenience we will assume $X$ to be a graph, which we can because every length space is quasi-isometric to a graph (\Or{qi graph}). 

Since $X$ has  infinite diameter, there is a sequence of geodesics $p_n:[0,n]\to X$ parametrised by arc length with $p_n(0)=a\in X$ a fixed point.
Call this set of geodesic paths $S$.
Fix $R>0, R\in \mathbb N$. We will show that $X$ contains either an $R$-fat $\bigvee P_n$
or an $R$-fat ray.

Consider the following property that we call $G(n)$: There is an infinite subset of $S$, say $S_n$, such that
ifor any $p_1,p_2\in S_n$ then $d(p_1(t_1),p_2(t_2))\leq R$ implies that $t_i\leq n$ for at least one $i (i=1,2)$. Note that
if $d(p_1(t_1),p_2(t_2))\leq R$ then $|t_2-t_1|\leq R$ as $p_1,p_2$ are geodesics.

Clearly if $G(n)$ holds for some $n\geq R$ then there is a subset of $S_n$
which is a $\bigvee P_n$ $R$-fat minor. Indeed it suffices to take a sequence $q_k\in S_n$
such that the length of $q_k$ is greater, say, than $2kR+n$. We set 
$$B_{v_0}=\bigcup q_k([0,n])$$
$B_{v_{ij}}=q_i(jR+n)$ for $j\geq 1, j\in \mathbb N$ to be the branch sets, where $j<i$. For $j=i$,  if $\ell=length(q_i)$ we set the branch set
$B_{v_{ii}}$ to be
$$B_{v_{ii}}=q_i([iR+n,\ell]).$$
We take $P_{vw}$  to be the obvious 
paths in the geodesics $q_k$ joining `successive' branch sets $B_v$, $B_w$. This defines an $\bigvee P_n$ $R$-fat minor.

So we may assume that $G(n)$
does not hold for any $n$.

Let $S'\subseteq S$ be infinite. We say that a path $p\in S'$ is a $bad(n)$ path if there are infinitely many paths $q\in S'$
such that $d(p(t),q(s))\leq R$ for some $t\geq n, s\in \mathbb N$.

Assume that for some $n$ there are no $bad(n)$ paths in some infinite subset of $S$, say $S'$. Pick then a path $q_1$ in $S'$ and take out from $S'$ the finite set
of paths $q$ such that $d(q_1(t),q(s))\leq R$ for some $t\geq n$ to obtain a set $A_1\subseteq S$. Now repeat, take $q_2\in A_1$
and eliminate from $A_1$ all paths $q$ such that $d(q_2(t),q(s))\leq R$ for some $t\geq n$ to obtain
a new infinite set of paths $A_2$.
 And so on. This produces an infinite set of paths $\{q_1,q_2,...\}$
that satisfies $G(n)$, a contradiction.  So there are $bad(n)$ paths in $S'$. In particular there are $bad(n)$ paths in $S$ for any $n$.

Let $q_1$ be $bad(n)$ in $S$ for $n=10R$ and let $C_1$ be the set of all paths $q$ such that $d(q_1(t),q(s))\leq R$ 
for some $t\geq n$. Since $q_1$ is a finite path there is some $k_1\in \mathbb N, k_1\geq n$ and an infinite
subset of $C_1$ (which we still denote $C_1$ to keep notation simple) so that for each $q\in C_1$
there is some $s\in \mathbb N$ for which
$d(q_1(k_1),q(s))\leq R$. Since both $q_1,q$ are geodesic paths $|s-k_1|\leq R$, hence
$$d(q_1(k_1),q(k_1))\leq 2R.$$

Since $C_1$ is an infinite subset of $S$, it contains a $bad(10k_1)$ path $q_2$. We repeat
and we obtain as before $k_2\in \mathbb N, k_2\geq 10k_1$ and an infinite subset $C_2$ of $C_1$ so that
so that for each $q\in C_2$ there is some $s$ so that
$d(q_2(k_2),q(s))\leq R$. As before for all $q\in C_2$ we have
$$d(q_2(k_2),q(k_2))\leq 2R.$$

We continue inductively constructing infinite subsets $S=C_0\supset C_1\supset C_2 \ldots$
and paths $q_1,q_2,\ldots$ such that $q_i\in C_{i-1}$  with the following property:
$d(q_i(k_i),q_j(k_i))\leq 2R$ for all $j>i$ where the sequence $\{k_i\}$ satisfies $k_{i+1}\geq 10k_i$, $k_1\geq 10R$.

We use this sequence of paths to construct an $R$-fat ray $r$ as follows. For $i\geq 1$ we pick a geodesic path $r_i$ joining $q_i(k_i)$ to $q_{i+1}(k_{i})$. We define the branch set $B_{v_i}$ to be the set 

$$q_i([k_i-2R,k_i])\cup r_i\cup q_{i+1}([k_{i},k_{i}+2R])$$

For $i=0$ we set $B_{v_0}=q_1(0)$

We define $P_{e_i}$ for $i\geq 1$ to be the arc
$$q_{i+1}([k_{i}+2R,k_{i+1}-2R]).$$

For $i=0$ we define $P_{e_0}=q_1([0,k_1-2R ])$.
It is clear then that by construction the set
$$r=\bigcup B_{v_i} \bigcup P_{e_i}$$ is an $R$-fat ray.
} 

\comment{
A potential approach towards \Cnr{Cnr Konig} is to take, \fe\ $K\in \N$, a $K$-net $G_n$ in $X$, and apply the following to $G_n$:
\begin{lemma} \label{lem}
Let $G$ be a (connected) graph of infinite diameter. Then $G$ contains the ray or $\bigvee P_n$ as a subgraph.
\end{lemma}
\begin{proof}
Pick a vertex $o\in V(G)$, and let $L_n,\nin$ be the set of vertices at distance exactly $n$ from $o$. \Fe\ $\nin$, pick a vertex $x_n\in L_n$, and a `downward' \pth{x_n}{o}\ $T_n$. By stopping $T_n$ upon reaching $\{T_j \mid j<n\}$ we can ensure that $\bigcup T_n$ is a (directed) tree $T\subseteq G$. We distinguish two cases. 

If $|L_n \cap T| <\infty$ holds \fe\ \nin, then K\"onig's lemma yields an infinite geodesic in \G. 

Otherwise, we can choose the smallest index  $n$ \st\ $|L_n \cap T| = \infty$. Then some $o'\in L_j, j< n$ sends infinitely many disjoint paths to $L_n$, and by extending these paths to higher $L_i$'s we obtain a copy of $\bigvee P_n$ in $T$.
\end{proof}

But I cannot yet use \Lr{lem} to prove \Cnr{Cnr Konig}.
}

\section{Coarse Halin Theorems} \label{sec Halin}

For a cardinal $k$, and a graph $H$, let $k \cdot H$ denote the disjoint union of $k$ copies of $H$. Let $R$ denote the ray, i.e.\ the 1-way infinite path. Halin \cite{halin65} proved that \fe\ graph $G$, if $n\cdot R \subset G$ holds \fe\ $\nin$, then $\omega \cdot R \subset G$. Does the coarse version hold?

\begin{conjecture} \label{Cnr HRT}
(Coarse Halin Ray Theorem) Let $X$ be a length space. Suppose $n\cdot R \prec^\infty X$ holds \fe\ \nin. Then $\omega \cdot R \prec^\infty X$.
\end{conjecture}

Let $\mathcal{H}$ denote the `half' square grid on $\N \times \N$, and let $\mathcal{F}$  denote the `full' square grid on $\Z \times \Z$. Halin's grid theorem \cite{halin65} says that every 1-ended graph satisfying $n\cdot R \subset G$ \fe\ \nin\ contains a subdivision of $\mathcal{H}$.

\begin{problem}
(Coarse Halin Grid Theorem) Let $X$ be an 1-ended length space. Suppose $\omega \cdot R \prec^\infty X$. Then $\mathcal{H} \prec^\infty X$.
\end{problem}

\begin{problem}
Let $G$ be a \Cg\ of an 1-ended finitely generated group. Must $\mathcal{H} \prec^\infty G$ hold? Must $\mathcal{F} \prec^\infty G$ hold?
\end{problem}

\section{Coarse Menger Theorem} \label{sec Menger}

In this section we prove the case $n=2$ of \Cnr{conj MM}. In \Sr{sec div rays} below we present a corollary for diverging rays in infinite graphs. We restate our special case for convenience: 

\begin{theorem} \label{Menger}
Let $X$ be a graph or a geodesic metric space, and $A,Z$ two subsets of $X$. For every $K>0$, there is either
\begin{enumerate}
\item \label{S} a set $S \subset X$ with $\diam(S)\leq K$ such that $X \setminus S$ contains no path joining $A$ to $Z$, or 
\item \label{paths} two paths joining $A$ to $Z$ at distance at least $a=K/16\cdot 17=K/272$ from each other.
\end{enumerate}
Moreover, \fe\ $x\in A, z\in Z$, we can choose the paths of \ref{paths} so that both $x,z$ are among their four endpoints.

\end{theorem}


We prepare the proof with some preliminaries. 
We will assume that \ref{S} does not hold and we will show that \ref{paths} holds.

We fix throughout the proof a shortest path $\gamma :[0,\ell]\to X $ joining $A$ to $Z$.
 
\begin{definition}[Bridges] \label{bridge}
A \textit{bridge} $B$ is a union of three successive paths $c_1\cup b\cup c_2$ 
such that:

\begin{itemize}
\item for any $x\in b, y\in \gamma $, we have  $d(x,y)\geq K/8$.

\item If $b\cap A\ne \emptyset $ then $c_1$ is an endpoint of $b$ and similarly if
$b\cap Z\ne \emptyset $ then $c_2$ is an endpoint of $b$; otherwise
$c_1,c_2$ are (geodetic) paths of length $K/8$ joining the endpoints
of $b$ to $\gamma $. 

\end{itemize}

\end{definition}

We call $c_1,c_2$ the \defi{legs} of $B$, and $b$ its \defi{spine}. The two endpoints of $B$ on $\gamma$, if they exist, are denoted by $B^0,B^1$ so that $B^0< B^1$. The \defi{interval} of $B$ is the interval $[B^0,B^1]$ of $\gamma$. If $c_1$ is a point (of $A$) then the interval of $B$ is $[\gamma (0),v]$ where $v$ is the
endpoint of $c_2$, and we define similarly the interval of $B$ when $c_2$ is a point of $Z$.

We say that a bridge is \textit{maximal} if its interval is not contained properly in the interval of another bridge.	
 
\begin{proposition} \label{max}
Let $B$ be a maximal bridge with spine $b$. For any point $x\in b$, let $x'$ be a point of $\gamma$ minimizing $d(x,x')$. Then $x'\in [B^0,B^1]$.
\end{proposition}
\begin{proof}
If $x'\not\in [B^0,B^1]$, say $x'>B^1$, then we can modify $B$ into a bridge $B'$ by adding an $x$--$x'$~geodesic and removing the subpath of $B$ from $x$ to $B^1$. The interval of $B'$ is $[B^0,x']$ which strictly contains $[B^0,B^1]$, contradicting the maximality of $B$.
\end{proof}

We assume that $\gamma $ is parametrized by arc-length. If $a\in \gamma $ and $r>0$, we denote by 
$a-r$  the point at distance $r$ from $a$ on $\gamma $ which is before $a$. We define 
 $a+r$ similarly.

\begin{definition}[Surrounding] \label{surround}
We say that a bridge $B$ $d$-\textit{surrounds} a point $x=\gamma (t)\in \gamma $ if the interval $[t_1,t_2]$ of $B$
contains $x-d$ and $x+d$. Here if $t<d$ or $t>\ell -d$ we set $x-d=\gamma (0)$ and $x+d=\gamma (\ell)$ respectively.
\end{definition}

We restrict our attention now to maximal bridges.

\begin{definition}[Order] \label{order}
Let  $B_1, B_2$ be maximal bridges. 
If $B_2^0 > B_1^0$ 
(in which case $B_2^1 > B_1^1$ holds too by maximality), we write $B_2>B_1$. 
\end{definition}

For $R\in \R$, we say that $B_2$ \defi{$R$-crosses}  $B_1$ if $B_2$ $R$-surrounds $B_1^1$. 

\begin{lemma} \label{bridge}
Assume condition \ref{S} of \Tr{Menger} fails. Then for any $x\in \gamma $ there is some bridge $B$ that $3K/8$-surrounds $x$. 
\end{lemma}
\begin{proof}
The ball $\ball_x(3K/8)$ separates $\gamma$ into two components $\gamma_A$ and $\gamma_Z$. Let $A':= A \cup \{y\in X \mid d(y, \gamma_A) \leq K/8 \}$, and define $Z'$ analogously. Let $\alpha$ be a shortest path in $X \sm \ball_x(K/2)$ joining $A\cup \gamma_A$ to $Z\cup \gamma_Z$, which exists by our assumption. Note that $\alpha$ contains at least one subpath $b$ from a point $p\in A'$ to a point $q\in Z'$ the interior of which is disjoint from $A' \cup Z'$. If $p\in A$ we let $c_1$ be the trivial path containing just $p$, otherwise we let $c_1$ be a geodesic from $p$ to $\gamma_A$. Define $c_2$ analogously.
 
We claim that $c_1 \cup b \cup c_2$ is a bridge that $3K/8$-surrounds $x$. Indeed, notice that each $c_i$ is either a singleton or has length $K/8$ and it starts outside $\ball_x(K/2)$, therefore it avoids $\ball_x(3K/8)$. Moreover, for every $z\in b$, we have  
$d(z, \gamma_A \cup \gamma_Z) \geq K/8$ and $d(z,\ball_x(3K/8))\geq K/8$. Since $\gamma\subseteq \gamma_A \cup \gamma_Z \cup \ball_x(3K/8)$, it follows that $d(z,\gamma )\geq K/8$, and our claim is proved. 

\end{proof}

\begin{lemma} \label{disjoint}

Let $B<C$ be maximal bridges with 
$C^0> B^1$.
Assume that there are $x\in B, y\in C$ such that $d(x,y)\leq K/4$. Then  $d(B^1,C^0)< 3K/4$.  
\end{lemma}
\begin{proof} Consider a shortest path $\alpha $ joining $x,y$, and let $|\alpha|\leq K/4$ denote its length. Let $x'\in[B^0,B^1]$, $y'\in [C^0,C^1]$ be closest points to $x,y$ respectively
on $\gamma$, where we used \Prr{max}. By the maximality of the bridges, there is a point $z$ on $\alpha $ such that $d(z,z')<K/8$ for some $z'\in \gamma $. Notice that $d(x,x') \leq d(x,z') < |x \alpha z| + K/8$, where $x \alpha z$ denotes the subpath of $\alpha$ from $x$ to $z$. Similarly, we have $d(y,y') < |z \alpha y| + K/8$.
Then $$d(x',y')\leq d(x,x')+d(x,y)+d(y,y')< 2K/8 + 2 |\alpha| \leq K/4 + 2K/4.$$
It follows that $d(B^1,C^0) < 3K/4$.
\end{proof}

We choose a smallest set $\mathcal U$ of maximal bridges such that any $x\in \gamma $ is $K/4$-surrounded by some 
bridge in $\mathcal U$, and 
\labtequ{length}{every $B\in \cu$ has an interval of length at least $3K/4$.} 
Such a \cu\ exists by \Lr{bridge}. We remark that $\mathcal U$ is finite because $\gamma$ is compact, and any bridge that $3K/4$-surrounds $x\in \gamma $ also $K/4$-surrounds any point in an open interval of $x$ on $\gamma$.

The minimality of $\mathcal U$ implies that the initial points of its bridges are not too dense on $\gamma$:

\begin{proposition} \label{U i}  
Let $B,C,D$ be bridges in $\mathcal U$. Suppose that
		 $C$ $r$-crosses $B$, and $D$ $r$-crosses $C$ for some $r\leq K/4$. Then $d(B^0,D^0)\geq r$. 
\end{proposition}
\begin{proof}
Suppose $d(B^0,D^0)< r$. This implies $D^0 \in [B^0,B^1]$ by \eqref{length}. We claim that any point that is $K/4$-surrounded by $C$ is also $K/4$-surrounded by either $B$ or $D$, contradicting
the minimality of $\mathcal U$. So suppose $p$ is $K/4$-surrounded by $C$. Then $d(D^1,p)> K/4$ as $D^1> C^1$. Moreover, we have $p>D^0$ because $d(C^0,p)\geq K/4$ and $d(C^0,D^0)< r$ by our assumptions.

We consider two cases. If $p>B^1$, then either $d(D^0,p)\geq K/4$, in which case $p$ is $K/4$-surrounded by $D$ and we are done, or $d(D^0,B^1)< d(D^0,p) < K/4$. In the latter case, however, the triangle inequality yields $d(B^0,B^1) \leq d(B^0,D^0) + d(D^0,B^1) < K/2$, contradicting \eqref{length}. 

It remains to consider the case where $p\leq B^1$. As above, we obtain  $d(D^0,p)< K/4$ unless $p$ is $K/4$-surrounded by $D$. Similarly, we have  $d(p,B^1)< K/4$ unless $p$ is $K/4$-surrounded by $B$. But then the triangle inequality yields $d(B^0,B^1) \leq d(B^0,D^0) + d(D^0,p) + d(p,B^1) < 3K/4$, again contradicting \eqref{length}. (This last inequality is the reason why we chose \cu\ so that any $x\in \gamma $ is $K/4$-surrounded, sacrificing some of the $3K/8$ provided by \Lr{bridge}.)
\end{proof}

A sequence $\cb= B_1, \ldots B_n$ of bridges is said to be \defi{$R$-crossing} for some $R\in \R$, if $B_1$ $R$-surrounds $\gamma(0)$, each $B_i, i\geq 2$ $R$-crosses $B_{i-1}$, and $B^1_n\in Z$. Note that $B_n$ must $R$-surround $\gamma(1)$ as $\gamma(1)> B^1_{n-1}$.
Such a sequence is \defi{perfect}, if in addition its initial points are $R$-apart, i.e.\ $d(B_i^0,B_j^0)\geq R$ holds \fe\ $1\leq i,j \leq n$.

\smallskip
\noindent {\it {\bf Remark:} A perfect $R$-crossing sequence \cb\ remains a perfect $R$-crossing sequence if we exchange the roles of $A$ and $Z$ and reverse the direction of $\gamma$ and the ordering of \cb. But we will not use this fact.}

\smallskip
It is straightforward to check that if we order the elements of \cu\ by $<$ then the resulting sequence $S$ is a $K/4$-crossing sequence. Our next lemma says that we can obtain a perfect subsequence by sacrificing half of our crossing constant.
\begin{lemma} \label{initials} 
Let $S= \{B_i\}_{1\leq i \leq n}$ be a $K/4$-crossing sequence of bridges.  Then there is a perfect $K/8$-crossing subsequence $S'$ of $S$. 
\end{lemma}
\begin{proof}
We claim that 
\labtequ{xyz}{there are no three  initial points $x < y < z$ of elements of $S$ \st\ $d(x,y), d(y,z) < K/8$.} Indeed, if such a triple exists, then we may assume that no other initial point of an element of $S$ lies in $[x,z]$ by choosing $x,y,z$ as close to each other as possible. It follows that there is $i\leq n$ \st\ $x= B_i^0, y= B_{i+1}^0, z= B_{i+2}^0$, and so $B_{i+1}$ $K/4$-crosses $B_i$ and $B_{i+2}$ $K/4$-crosses $B_{i+1}$. This contradicts \Prr{U i}, and so our claim \eqref{xyz} is proved.

Thus the pairs $\{x,y\}$ of initial points of elements of $S$ with $d(x,y) < K/8$ form a matching, i.e.\ no point participates in more than one such pair. For every such pair with $x<y$, remove the bridge $B_x$ with initial point $x$ from $S$, and let  $S'$ be the remaining subsequence of $S$. Clearly, we cannot have a pair $\{x,y\}$ as above in $S'$. To see that $S'$ is $(R-K/8)$-crossing, notice that whenever we removed a bridge $B_x$ from $S$, we left a bridge $B_y$ in $S'$ such that $B_y^0 < B_x^0 + K/8$ and  $B_y^1 > B_x^1$. Thus any $p\in \gamma$ that was $R$-surrounded by $B_x$ is $(R-K/8)$-surrounded by $B_y$.
\end{proof}

\begin{proposition} \label{U ii}   
Suppose that $B_1,B_2,B_3,B_4,B_5$ are bridges 
		such that each $B_i, i\geq 2$, $K/8$-crosses $B_{i-1}$, 
		and $B_i^0$ lies in the interval of $B_1$ for every $2\leq i \leq 5$.  Then at least one $B_i$ does not lie in $\mathcal U$.
\end{proposition}
\begin{proof}
We will show that any point $K/8$-surrounded by $B_4$ is $K/8$-surrounded by either $B_2$, $B_3$, or $B_5$, so the result follows by the minimality of $\mathcal U$. 
 
For this, notice that our assumptions impose the ordering 
$$B_1^0 <  B_2^0 <  B_3^0 < B_4^0 < B_5^0 < B_1^1 < B_2^1< B_3^1 < B_4^1 < B_5^1.$$
The definition of crossing implies that  $d(B_{i-1}^1, B_i^j) \geq K/8$ for every $i\in \{2,5\}$ and $j\in \{0,1\}$.
Moreover, applying \Prr{U i} twice, once on $B_2,B_3,B_4$, and once on $B_3,B_4,B_5$, we also obtain $d(B_2^0,B_4^0), d(B_3^0,B_5^0)\geq K/8$. 

\begin{figure} 
\begin{center}
\begin{overpic}[width=1\linewidth]{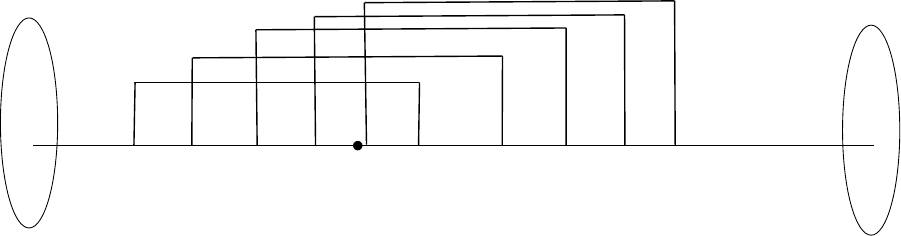} 
\put(2,26){$A$}
\put(96,25){$Z$}
\put(39,7){$p$}
\put(12,7){$B_1^0$}
\put(20,7){$B_2^0$}
\put(27,7){$B_3^0$}
\put(33,7){$B_4^0$}
\put(44,7){$B_1^1$}
\put(54,7){$B_2^1$}
\put(61,7){$B_3^1$}
\put(68,7){$B_4^1$}
\put(73.5,7){$B_5^1$}
\put(83,12){$\gamma$}
\end{overpic}
\end{center}
\caption{One of the possibilities in the proof of \Prr{U ii}.} \label{figBridges}
\end{figure}

Suppose now that a point $p$ is $K/8$-surrounded by $B_4$, and so $p\in [B_4^0,B_4^1]$. 
If $p\in [B_4^0,B_1^1]$ (\fig{figBridges}), then it is $K/8$-surrounded by $B_2$, because $d(B_2^0,B_4^0), d(B_1^1,B_2^1) \geq K/8$.
If $p\in [B_1^1,B_2^1]$, then it is $K/8$-surrounded by $B_3$, because $d(B_2^1,B_3^1) \geq K/8$ and $d(B_3^0,B_1^1)> d(B_3^0,B_5^0) \geq K/8$.
Finally, if $p\in [B_2^1,B_4^1]$, then it is $K/8$-surrounded by $B_5$, because $d(B_5^0,B_2^1)> d(B_0^1,B_2^1) \geq K/8$ and $d(B_4^1,B_5^1)\geq K/8$. This proves our claim that $p$ is always $K/8$-surrounded by either $B_2$, $B_3$, or $B_5$, contradicting the minimality of \cu.
\end{proof}

\begin{lemma} \label{U iii} 
Suppose that $B_i:i=1,\ldots,n$ are bridges in $\mathcal U$ such that $B_i$ crosses $B_{i-1}$ for $i=2,\ldots,n$, and that 
		there are $x\in B_1, y\in B_n$ such that $d(x,y)<K/4$.  Then $n\leq 8$.
\end{lemma}
\begin{proof}
Suppose $n\geq 9$. By \Prr{U ii}, the intervals of $B_1,B_5$ are disjoint. By the same argument, the intervals of $B_5,B_n$ are disjoint too. By lemma \ref{disjoint}, we have $d(B_1^1,B_9^0)< 3K/4$. These facts combined imply that the interval of $B_5$ has length less than $3K/4$,  contradicting \eqref{length}. 
\end{proof}

Our plan is to construct the two paths $\alpha_1,\alpha_2$ joining $A$ to $Z$ required by \Tr{Menger} by combining subpaths of $\gamma$ with bridges from a sequence $S'$ as in \Lr{initials}. However, we can have bridges $B,C$ in $S'$ arbitrarily close to each other (but not intersecting, as this would contradict maximality), in which case we have to avoid putting $B$ in $\alpha_1$ and $C$ in $\alpha_2$. Therefore, if $d(B,C)<a$, we will use a \pth{B}{C} $\alpha$ of length less than $a$ to join $B,C$, and produce a new path called a \defi{\mb}. We may need to iterate the process, joining two nearby \mb s into a new \mb. But the following lemma will exploit \Lr{U iii} to upper bound the complexity of \mb s. 

Formally, we can simply define a \defi{\mb} $B$ to be a path
with endpoints in $\gamma \cup A \cup Z$ whose interior is disjoint from $\gamma \cup A \cup Z$. But the \mb s we will actually work with will obey further restrictions, being constructed recursively starting from bridges as explained above. Much of the notation we introduced for bridges can be applied verbatim to \mb s, and we will do so without further comment.

A \defi{join} of a \mb\ $D$ is a maximal subpath of $D$ consisting of points that lie on no bridge of $S'$. 

Choose $a=  K/(16\cdot 17)$. 
\begin{lemma} \label{meta} 
There is a perfect, $K/8$-crossing, sequence $M$ of \mb s \st\
\begin{enumerate}
\item \label{m ii} each join $J$ of an element of $M$ has length $\ell(J) < 15a$; 
\item \label{m iii} we have $d(J, \gamma)\geq a$ for each join $J$ of an element of $M$, and
\item \label{m i} if $x,y$ are points of distinct \mb s $B,C$ of $M$, then $d(x,y)\geq a$ unless there are endpoints $B^i,C^{1-i}, i\in \{0,1\}$ with $d(B^i,C^{1-i})<K/8$. 
\end{enumerate}
\end{lemma}
\begin{proof}
We start with a sequence $S'$ of maximal bridges as in \Lr{initials}, and construct $M$ recursively by joining (meta-)bridges together as long as \ref{m i} is not satisfied. Since $S'$ is finite, the process will terminate and \ref{m i} will then be satisfied. The fact that \ref{m ii} and \ref{m iii} hold will then be deduced from \Lr{U iii}.
\smallskip

To make this precise, set $M_0:= S'$, and proceed inductively as follows. At step $i=1,2,\ldots$, we have defined a sequence $M_{i-1}$ of \mb s. If $M_{i-1}$ satisfies \ref{m i}, we set $M:= M_{i-1}$ and stop. Otherwise we can find points $x,y$ in distinct elements $B, C$ of $M_{i-1}$, \st\ $d(x,y)<a$ and $d(B^i,C^j)\geq K/8$ \fe\ $i,j\in \{0,1\}$. We pick a shortest \pth{x}{y}\ $\alpha$, and combine it with a subpath of each of $B,C$ to form a \mb\ $D$ as follows. We notice that $\alpha$ separates $B\cup C$ into 4 subpaths, and choose two of them $B_\alpha,C_\alpha$ so as to maximize the interval of the \mb\ $D:= B_\alpha \cup \alpha \cup C_\alpha$. Notice that $D$ now $K/8$-surrounds any point of $\gamma$ previously $K/8$-surrounded by either $B$ or $C$. Thus the sequence $M_i$ obtained by replacing the subsequence of $M_{i-1}$ between $B$ and $C$ by $D$ is a perfect $K/8$-crossing sequence. For later use, we colour the path $\alpha$ \defi{brown}.

\medskip
The final sequence $M$ of this process satisfies  \ref{m i} by definition. To prove that it satisfies  \ref{m ii}, we will introduce some terminology. 

A \defi{join-tree} is any maximal (with respect to inclusion) connected subspace of $X$ all points of which had been coloured brown in some step $i$ of the above process. It is easy to see, by induction on $i$, that every join-tree is indeed homeomorphic to a finite simplicial tree, and that every join of an element of $M$ is contained in some join-tree. The \defi{length $\ell(T)$} of a join-tree $T$ is its 1-dimensional Hausdorff measure, and it equals the sum of the lengths of its constituent paths. We will prove \ref{m ii} by showing that $\ell(T)< 15a$ holds \fe\ join-tree $T$.

It is easy to see ---by induction on $i$--- that every leaf  of a join-tree $T$ lies on a bridge of $S'$. The \defi{rank $\rank(T)$} of $T$ is the number of leaves of $T$. Note that every two leaves of $T$ lie in distinct elements of $S'$, because when a new brown path is introduced it joins distinct \mb s. (Such a leaf $P$ is a (trivial) join-tree itself, and we have $\rank(P)=1$.)
 
We claim that  \fe\ join-tree $T$ with $\rank(T)\geq 2$, we have
\labtequ{tree length}{$\ell(T)< a (\rank(T)-1)$.} 
We prove this by induction on $\rank(T)$. Let $\alpha$ be an edge of $T$, and recall that $\ell(\alpha)< a$. The case $\rank(T)=2$ occurs exactly when $T=\alpha$, and so our claim is corroborated. For the general case, notice that $T \sm \alpha$ has two components $T_1,T_2$. We have $\rank(T)= \rank(T_1)+  \rank(T_2)$, and $\ell(T)< \ell(T_1) + \ell(\alpha) + \ell(T_2)$ and so our inductive hypothesis yields  $\ell(T)< a (\rank(T_1)+  \rank(T_2)-2)+a = a (\rank(T)-1)$ as claimed.

Next we claim that 
\labtequ{tree rank}{there is no join-tree $T$ with $\rank(T)\geq 17$.}
Indeed, if such a $T$ exists, then by recursively removing its last edge coloured brown, we can obtain a join-subtree $T'$ with $16\geq \rank(T')\geq 9$ by the pigeonhole principle. Let $B_i,B_j$ be bridges of $S'$ incident with $T'$, \st\ $j\geq i+8$, which exist since $\rank(T')\geq 9$. By \eqref{tree length} we have $\ell(T')< 16a$, and so $d(B_i,B_j)< 16a$. This contradicts \Lr{U iii} because $16a\leq K/4$, and so \eqref{tree rank} is proved. \mymargin{need $a\leq K/32$} Combining \eqref{tree rank} with \eqref{tree length} yields \ref{m ii}.

\medskip
To prove \ref{m iii}, we will check that 
\labtequ{leaf}{every join-tree $T$ has a leaf $p$ \st\ $d(p,\gamma)\geq K/16-a$.}
For this, let $\alpha$ be the first subpath of $T$ coloured brown, let $x,y$ be its two endpoints, and let $B_x,B_y$ be the bridges in $S'$ containing them. Thus both $x,y$ are leaves of $T$. Since the spines of $B_x,B_y$ are at distance at least $K/8$ from $\gamma$, and $\ell(\alpha)<a$, the only interesting case is where both $x,y$ lie on legs. Let $x',y'$ be the endpoints of those legs on $\gamma$, and recall that we must have $d(x',y')\geq K/8$ for otherwise we would not have coloured $\alpha$ brown. Thus the triangle inequality yields $d(x,\gamma) + d(y,\gamma) + a \geq K/8$. Moreover, we have $d(y,\gamma) \leq d(x,\gamma) + a$. Combining these inequalities we obtain $2d(x,\gamma) + 2a \geq K/8$, and letting $p:=x$ proves \eqref{leaf}.

Combining  \ref{m ii} and \eqref{leaf} we deduce that every join-tree lies at distance at least  $K/16-a-15a=K/16-16a$ from $\gamma$. Since $17a\leq K/16$, and every join is contained in a join-tree, this implies \ref{m iii}. \mymargin{need $a\leq K/16\cdot 17$}
\end{proof}

\noindent {\bf Remark:} The last inequality is the bottleneck for maximizing the distance $a$ between the two paths $\alpha_i$ of \Tr{Menger}.

\bigskip

\comment{
\begin{proposition} \label{}  
Let $x,y$ be points of meta-bridges $X,Y$ \st\ $d(x,\gamma), d(y,\gamma)\leq r$. Then $d(x,y)\geq a$ unless there are endpoints $x',y'$ of $X,Y$ \st\ $d(x',y')\leq ...$
\end{proposition}
\begin{proof}
\end{proof}
}

We have collected all the necessary ingredients for the proof of our \Tr{Menger}.

\begin{proof}[Proof of \Tr{Menger}]
Assuming condition \ref{S} fails, we will now define the two paths $\alpha _1, \alpha _2$ of \ref{paths}, joining $A$ to $Z$, and satisfying $d(\alpha _1,\alpha _2)\geq a$. \smallskip

Each $\alpha_j$ will consist of intervals of $\gamma$ and \mb s in a family $M$ as in \Lr{meta}. It will be convenient to colour the paths contained
in $\alpha _1$ red and the paths of $\alpha _2$ green. We will define these paths inductively. 

Since $M$ is a perfect sequence, it starts with a \mb\ $M_1$ surrounding $t_0=\gamma (0)$. We colour $M_1$ red, and set $\beta_1^1:= M_1$. We colour the initial subpath of $\gamma$ from $t_0$ to $t_1:= M_1^1$ green, and  set $\beta _2^1:= [t_0,t_1]$. This completes step 1 of our induction. 
 
For step $n$, assume that we have already defined a red path $\beta_1=\beta^{n-1}_1$ and a green path $\beta_2=\beta^{n-1}_2$ with a common endpoint $t= t_{n-1}\in \gamma$. Let $C\in M$ be a \mb\ that $K/8$-surrounds $t$. Let $s$ be a point of $\gamma \cap (\beta_1 \cup \beta_2)$ closest to $C^0$ (it is possible that $s=C^0$). Suppose $s\in \beta_i$. Remove the subpath of $\beta_i$ from $s$ to $t$, and replace it by $C \cup I$ where $I$ is the subpath of $\gamma$ between $s$ and $C^0$\mymargin{ (\fig{fig betas})}. Finally, extend the other path $\beta_{3-i}$ by the subpath of $\gamma$ from $t$ to $C^1$, if $C^1\in \gamma$, or from $t$ to $\gamma (1)$ if $C^1\in Z$.
In the latter case we terminate the process and set $\alpha_j:= \beta_j, j\in \{1,2\}$. This completes the construction of  $\alpha _1, \alpha _2$.

\medskip
It remains to check that $d(\alpha _1,\alpha _2)\geq a$, and to do so we first point out some properties of our construction.

Firstly, every \mb\ $C\in M$ contained in an $\alpha _1$ is traversed from $C^0$ to $C^1$ as $\alpha _1$ moves from $A$ to $Z$. 

Secondly, any \mb\ $B$ removed from $\beta_1 \cup \beta_2$ in some step $n$ cannot return to $\beta_1 \cup \beta_2$ in a later step $n'>n$, because $B^1\leq t_{n-1}$.

Finally, notice that if $d(C^0,B^1) < K/8$ holds for some $\mb s$ $B,C\in M$,  then either $s= B^j$, in which case $B$ and $C$ become part of the same $\beta_i$,  
or $s\neq B^j$ and $B$ is not part of $\beta_1 \cup \beta_2$ at step $n$, hence nor at any later step. To summarize, we have proved that 
\labtequ{BC}{if $d(C^i,B^{i'}) < K/8$ holds for some \mb s $B,C\in M$, then $B,C$ cannot lie in distinct $\alpha_j$.}
Define a \defi{\game} to be a shortest interval of $\gamma$ whose endpoints are both endpoints of \mb s of $M$. Thus \eqref{BC} implies
\labtequ{game}{if two \game s $e,f$ are at distance $< K/8$, then $e,f$ cannot lie in distinct $\alpha_j$.}
%

We now check that $d(x,y)\geq a$ holds \fe\ $x\in \alpha_1, y\in \alpha _2$, distinguishing three cases according to whether these points lie on \mb s or on $\gamma$.

If $x,y$ lie on  \mb s $B,C$, then $B\neq C$ by construction. Then \eqref{BC} says that the endpoints of $B,C$ are at distance at least $K/8$ from each other, and so property \ref{m i} of \Lr{meta} yields $d(x,y)\geq a$.

If $x,y\in \gamma$, then \eqref{game} implies again $d(x,y)\geq a$.

Finally, if  $x \in B \in M$, and $y\in \gamma$, then property \ref{m iii} of \Lr{meta} yields $d(x,y)\geq a$ unless $x$ lies on a leg of $B$. Here, a \defi{leg} of a \mb\ is a maximal initial or final subpath contained in a bridge in $S'$. If $x$ lies on a leg $c$ of $B$, we recall that if  $c$ lies in $\alpha_j$, then one of the two \game s $e$ incident with its endpoint $x'\in \gamma$ also lies in $\alpha_j$. Since $y$ cannot lie in $e$, \eqref{game} yields $d(y,x')\geq K/8$. We also know that $d(x,y)\geq d(x,x')$ by the definition of a leg, and so  $d(x,y)\geq a$ easily follows.

This completes the proof of \Tr{Menger}. 
\end{proof}



\subsection{Alternative construction of the paths --- reducing to the classical Menger theorem}

This section offers an alternative proof of \Tr{Menger}, which is simpler but gives a weaker lower bound $a$ on $d(\alpha_1,\alpha_2)$. The main difference is that instead of constructing $\alpha_1,\alpha_2$ explicitly, we apply Menger's theorem to an auxiliary graph, consisting of \game s and \mb s in $M$.

For this, we first adapt \ref{m i} of \Lr{meta} by lowering the constant $K/8$ to $K/20$, say, and lowering $a$ proportionately; the proof can be repeated verbatim after this change. This change allows us to deduce the following corollary of \Lr{meta}: 
\begin{corollary} \label{cor meta}
If $B,C$ are \mb s in $M$ or \game s, then $d(B,C)\geq a$ unless $B,C$ are incident or as in \ref{m i} of \Lr{meta} (i.e.\ they have $K/20$-close end-points).
\end{corollary}
\begin{proof}
If $B,C\in M$, then this is just \Lr{meta} \ref{m i}. If $B,C$ are both \game s, then it is obvious because $\gamma$ is a geodesic. Otherwise, we have $B\in M, C\subseteq \gamma$ or vice-versa. In this case we use \Lr{meta} \ref{m iii}, that fact that spines of \mb s are far from $\gamma$, and that legs of \mb s are geodetics to $\gamma$.
\end{proof}

\begin{proof}[Alternative proof of \Tr{Menger}]
Let \g be the graph whose edge set consists of $M$ (as in \Lr{meta}) and the \game s that $M$ defines. \Fe\ $x,y\in V(G)$ joined by a \game\ $e$ of length less than $K/20$ (the new constant of \ref{m i} of \Lr{meta}), contract $e$. Let $G'$ be the graph obtained from \g after all these contractions. Notice that the contracted edges of \g are disjoint, and therefore  $G'$ has no cut-vertex $x$ since $M$ contains a \mb\ surrounding $x$. Let $\alpha_1,\alpha_2$ be two disjoint \pths{A}{Z}, which exist by Menger's theorem. Then $d(\alpha_1,\alpha_2)>a$ by \Cr{cor meta}.
\end{proof}

\subsection{Choosing endpoints}
We remark that $\gamma(0),\gamma(1)$ are among the four endpoints of $\alpha_1, y\in \alpha _2$ in the above construction. This can be seen directly from the construction, or by noticing that $\cu$, and hence $S'$ and $M$, contains at most one bridge with an endpoint in each of $A$ and $Z$ by maximality. Recall that $\gamma$ was chosen as an arbitrary shortest \pth{A}{Z}. By modifying $X$ we can strengthen the above conclusion:

\begin{corollary} \label{endpoints}
For every $x\in A$ and $z\in Z$, we can choose the paths of \Tr{Menger} so that both $x,z$ are among their four endpoints.
\end{corollary}
\begin{proof}
By the above remark, the statement holds if $d(x,z)= d(A,Z)$. We will extend $X$ into a geodesic metric space $X'$ and define $A'\ni x, Z' \ni z$ so that $d(x,z)= d(A',Z')$ and $X'$ has no subset $S$ with $\diam(S)\leq K$ separating $A'$ from $Z'$.             

To do so, \fe\ point $p \in A \sm \{x\}$, attach a geodetic arc $A_p$ of length  larger than $d(x,z)$ to $p$, and let $p'$ denote the other endpoint of $A_p$. Let $A':= \{x\} \cup \{p' \mid p \in A, p\neq x\}$. Let $Z':= Z$. There is a natural way to extend the metric of $X$ to $X'$, and this is also the unique metric $d$ that makes $X'$ a geodesic metric space. Note that $d(x,z)= d(A',Z')$ holds.

Let $S$ be a subset of $X'$ separating $A'$ from $Z'$ with $\diam(S)\leq K$. We modify $S$ into a subset $S'$ of $X$ by `projecting' $S \sm X$ onto $A$. To make this precise, let $S_A \subset A$ be the set of those $p\in A$ \st\  $S$ intersects $A_p$, and let $S':= S_A \cup (S\cap X)$. It \istc\ that $\diam(S')\leq  \diam(S)$, and that $S'$ separates $A$ from $Z$ in $X$. It follows that no such $S \subset X'$ can exist by our assumption on $X$. Applying \Tr{Menger}, and the above remark, to $X',A',Z'$ we obtain two \pths{A'}{Z'}\ that are $a$-far apart and contain $x,z$ among their four endpoints. Removing a subpath of the form $A_p$ from one of them completes our proof.
\end{proof}

\subsection{An application --- diverging rays} \label{sec div rays}

It is not hard to extend \Tr{Menger} to the situation where $Z$ is replaced by $\infty$; that is, we have

\begin{corollary} \label{cor inf Menger}
Let $X$ be a \lf\ metric graph, 
 and $A \subset X$. For every $K>0$, there is either
\begin{enumerate}
\item \label{S} a set $S \subset X$ with $\diam(S)\leq K$ such that $X \setminus S$ contains no path of infinite length starting at $A$, or 
\item \label{paths} two paths of infinite length starting at $A$ at distance at least $a=K/16\cdot 17$ from each other.
\end{enumerate}
\end{corollary}
\begin{proof}
Apply \Tr{Menger} to a sequence of spaces $\seq{X}$, namely the balls of radius $n \in \N$ around $A$ in $X$, with $Z=Z_n$ being the set of points at distance exactly $n$ from $A$. This yields pairs of \pths{A}{Z_n} at distance at least $a$ from each other, and we can use a compactness argument to obtain \pths{A}{\infty}  at least $a$ from each other.
\end{proof}

\noindent {\bf Remark:} the local-finiteness condition is necessary in \Cr{cor inf Menger}. \mymargin{added this remark} A non-locally-finite counterexample can be obtained as follows. Let $R=x_0x_1\ldots$ be  a ray, 
and for each $i>1$, join $x_0$ to $x_i$ by a path $P_i$ of length $i$, so that $P_i\cap P_j = \{x_0\}$ \fe\ $i\neq j$, to obtain a graph $X$ (in which $x_0$ is the only vertex of infinite degree, see \fig{figNLF}). Let $A$ be any infinite set of vertices of $X$. 
Easily, no set of finite diameter separates $A$ from infinity in $X$, i.e.\ \ref{S} fails. Moreover, if $X$ contains two infinite paths, then so does $X':=X-x_0$. But $X'$ is an 1-ended tree. Thus \ref{paths} fails as well.

\begin{figure} 
\begin{center}
\begin{overpic}[width=.45\linewidth]{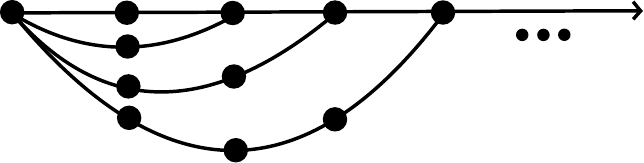} 
\put(-2,28){$x_0$}
\put(19,28){$x_1$}
\put(60,5){$P_4$}

\end{overpic}
\end{center}
\caption{A non-locally-finite graph failing the conclusion of \Cr{cor inf Menger}.} \label{figNLF}
\end{figure}

\medskip
We say that two rays $R,L$ in a graph \g  \defi{diverge}, if
\fe\ $n\in \N$ they have tails $R' \subseteq R, L' \subseteq L$ satisfying $d(R',L')>n$.

\begin{theorem} \label{thm div rays}
Let \g be a graph with bounded degrees which has an infinite set of pairwise disjoint rays. Then \g has a pair of diverging rays. 
\end{theorem}

The idea is to apply \Cr{cor inf Menger}, letting the scale $K$ increase as we progress towards infinity. Instead of increasing $K$, we will use the equivalent but more convenient method of scaling down the edge-lengths as the distance from a fixed origin tends to infinity. This is the role of the following lemma:

\begin{lemma} \label{lem balls}
Let $G=(V,E)$ be a graph with bounded degrees which has an infinite set of pairwise disjoint rays. Then \ti\ a finite vertex-set $A\subset V$ and an assignment of edge-lengths $\ell : E \to \R_{>0}$ with the following properties: 
\begin{enumerate}
\item \label{non sep} no ball of radius $1$ in the corresponding metric $d_\ell$ separates $A$ from $\infty$, and
\item \label{to zero} $\lim_{e\in E} \ell(e) = 0$.\footnote{This means that \fe\ $\epsilon>0$ all but finitely many $e\in E$ satisfy $\ell(e)<\epsilon$.}
\end{enumerate}
\end{lemma}

Let us first see how this implies \Tr{thm div rays}.

\begin{proof}[Proof of \Tr{thm div rays}]
Applying \Cr{cor inf Menger} to $G$ endowed with $d_\ell$ as in \Lr{lem balls}, we obtain a pair of rays $R,L$ starting at $A$ \st\ $d_\ell(R,L)>b$ for some (universal) constant $b>0$. Let $E_n:= \{e\in E \mid \ell(e)\geq 1/n\}$. Thus by \ref{to zero}, $E_n$ is finite. Let $E'_n$ be the set of edges of $G$ at distance at most $n$ from $E_n$ \wrt\ $d_G$, which is finite too since \g is \lf. To see that $R,L$ diverge, choose tails $R',L'$ avoiding $E'_n$. 
Notice that any  \pth{R'}{L'}\ $P$ either avoids $E_n$ or contains at least $n$ edges in $E'_n \sm E_n$. The fact that $\ell(P)>b$ implies that $P$ contains at least $\min\{n,bn\}$ edges, proving that $R,L$ diverge.  
\end{proof}

It remains to prove \Lr{lem balls}.

\begin{proof}[Proof of \Lr{lem balls}]
Let \seq{R} be a sequence of pairwise disjoint rays in \G. Let $k$ be an upper bound on the vertex-degrees of $G$. Fix a vertex $o\in V$. For $n=1,2,\ldots$, choose $r_n \in \N$ \leth\ \begin{enumerate}
\item \label{meets} $B_o(r_n)$ meets at least $k^{n+3}+2$ of the $R_i$, and 
\item \label{plus n}  $r_n > r_{n-1} + n$,
\end{enumerate}
where we set $r_0:= 1$, say. Let $S_1:= B_o(r_1)$, and for $n=2,3, \ldots$ let $S_n:=B_o(r_n) \sm S_{n-1}$. Assign length $\ell(e)= 1/n$ for each edge $e$ in $S_n,n\geq 1$. 

Clearly, this assignment satisfies property \ref{to zero} of the statement as each $S_n$ is finite. We claim that $\ell$ also satisfies \ref{non sep} with $A:= \bigcup_{1\leq i \leq 4}S_i$. To see this, pick any $x\in S_n$, and let $B$ denote its ball of radius 1 \wrt\ $d_\ell$. Thus we have to show that \ti\ a ray in $G \sm B$ starting in $A$. 

It follows easily from \ref{plus n} that 
\labtequ{B in Sn}{$B \subseteq S_{n-1} \cup S_n$ or  $B \subseteq S_n \cup S_{n+1}$.} 
Thus for any $y\in B$ we have $d_G(x,y)\leq n+1$ because $d_\ell(x,y)\leq d_G(x,y)/(n+1)$ by the previous remark. This means that $B$ has radius at most $n+1$ in $d_G$, and so it contains at most $k^{n+1}+1$ vertices since \g has maximum degree $k$. We consider two cases: if $n\geq 3$, then by \ref{meets}, $B$ avoids at least one ray $R_i$ meeting $S_{n-2}$. Joining that ray to $A$ by a shortest possible path ---which avoids $B$ by \eqref{B in Sn}--- we obtain a ray starting in $A$ and avoiding $B$. If $n< 3$, then again by \eqref{B in Sn} we can find a ray starting in $S_4\subset A$ and avoiding $B$. This proves that \ref{non sep} is satisfied.
\end{proof}

\section{Further problems} \label{sec OP}
\pp{added questions and made comments in the next couple of paragraphs}
As we saw in the introduction, \Cnr{conj fat min} can be very difficult even for a simple graph $H$. The following special case is of particular interest (\Cnr{conj fat min} can be formulated with $H$ replaced by a finite set of finite graphs):
\begin{conjecture}[Coarse Kuratowski--Wagner theorem] \label{Kur}
A length space $X$ is quasi-isometric to a planar graph \iff\ it has no asymptotic $K_5$ or $K_{3,3}$ minor. 
\end{conjecture}
Here, and more generally in \Cnr{conj fat min}, we may have to allow the resulting planar graph to be edge-weighted; we do not know whether every planar graph with arbitrary real-valued edge-lengths (used to induce a length space) is quasi-isometric to a planar graph with all edge-lengths equal to 1. 

This special case is well-motivated in itself due to widespread interest in planar graph metrics, and due to the interest in groups acting on planar surfaces \cite{Kleinian}. 

Another interesting class of graphs is that of graphs of finite tree width.

\begin{conjecture}[Coarse grid theorem] \label{con grid}
A length space $X$ is quasi-isometric to a graph of finite tree width \iff\ it does not have the $n\times n$  grid as an   asymptotic minor for some \nin. 
\end{conjecture}
One can also ask whether under the same hypothesis $asdim (X)\leq 1$ (as this is weaker). 

Our \Cnr{conj fat min}  would imply that the fundamental Robertson-Seymour graph minor theory has a coarse geometric generalisation. One can of course formulate weaker versions of our conjecture in the same direction. For example, one could ask whether excluding an infinite family of asymptotic minors is equivalent to excluding a finite such family. In geometry and topology one often asks for the `nicest' metric among all metrics in a homeomorphism class. This is the classical topic of uniformisation of surfaces,
and this point of view was crucial in the recent classification of 3-manifolds after Thurston and Perelman.  \Cnr{conj fat min}  can be seen as a coarse analog to
uniformisation: we ask for a `nice' representative in a quasi-isometry class of graphs.

\medskip

Our next conjecture is a coarse version of a result of Thomassen  \cite{ThoHad} saying that an 1-ended vertex transitive graph is either planar or has an infinite clique as a minor. 

\begin{conjecture} \label{Tho}
Let $G$ be a locally finite, vertex-transitive graph. Then either $G$ is quasi-isometric to a planar graph, or it contains every finite graph as an asymptotic minor.
\end{conjecture}

Recall that a graph \g is \defi{$k$-planar}, if it can be drawn in $\R^2$ so that each of its edges crosses at most $k$ other edges.
\begin{problem} \label{k planar}
Let $G$ be a locally finite, vertex-transitive graph. Is it true that $G$ is quasi-isometric to a planar graph \iff\ it is $k$-planar for some $k\in \N$?
\end{problem}
Vertex-transitivity is essential here, for otherwise we could let \g be a sequence of subdivisions of $K_5$.

Our next problem is motivated by the following theorem of Khukhro: 
\begin{theorem}[\cite{KhuCha}] \label{Khukhro}
A finitely generated, infinite group $\Gamma$ is virtually free, if and only if for every finitely generated Cayley graph \g of $\Gamma$, there is some finite graph that is not a minor of \G.
\end{theorem}

\begin{conjecture} \label{virt planar}
A finitely generated group $\Gamma$ is virtually planar, if and only if for every finitely generated Cayley graph \g of $\Gamma$, there is some finite graph that is not an asymptotic minor of \G. 
\end{conjecture}
\defi{Virtually planar} here means that $\Gamma$ has a finite index subgroup $H$ which has a planar \Cg.  The forward implication of this is easy: by the Svarc--Milnor lemma, any finitely generated \Cg s of  $H$ and $\Gamma$
 are quasi-isometric, and so the statement follows from \Or{invariance}. (The analogous implication of \Tr{Khukhro} was proved `by hand' by Ostrovskii \& Rosenthal \cite{OstRosMet}.)

\medskip
As already mentioned, we expect \Cnr{conj MM} to be very difficult for $n\geq 3$. The following is a related algorithmic question. Let MM3 denote the algorithmic problem that takes as input a graph $G$ on $n$ vertices and two subsets  $A,Z$ of its vertex set, and asks whether there is a triple of \pths{A}{Z}\ that are pairwise at distance at least 3. 
\begin{conjecture} \label{conj MM NPc}
MM3 is NP-complete.
\end{conjecture}

A classical result of \Erd\ \& P\'osa \cite{ErdPosInd} says that every finite graph has either a $k\cdot K_3$ minor or a set of at most $f(k)$ vertices the removal of which results into a forest, where $f(k) \sim k \log k$ is universal. The following proposes a coarse analogue, motivated by both \Cnr{conj fat min} and \Cnr{conj MM}:

\begin{conjecture}[Coarse \Erd--P\'osa Theorem] \label{Cnr EP}
 There is a function $f: \N \to \N$ and a universal constant $C$, \st\ the following holds \fe\ length space  $X$.
If $n \cdot K_3$ is not a $K$-fat minor of $G$, then \ti\ a set $S$ of at most $f(n)$ points of $X$ \st\ the ball of radius $C K$ around $S$ meets all $K$-fat $K_3$ minors of $X$. (Alternatively, $X$ is a graph, and removing $B_{CK}(S)$ results into a quasi-forest.)
\end{conjecture}

\section{Latest progress}
Many of the questions of this paper have been answered between the original submission and the revision. The aim of this section is to provide an update.

Nguyen, Scott \& Seymour \cite{NgScSeCou} found  a counterexample to \Cnr{conj MM}. This was used by Davies, Hickingbotham,  Illingworth \& McCarty \cite{DHIM} to disprove \Cnr{conj fat min}. \Cnr{conj MM NPc} has been proved by Balig\'acs \& MacManus \cite{BalMcMmet}.

The case $H=K_4$ of  \Cnr{conj fat min} has been proved by Albrechtsen, Jacobs,  Knappe, \& Wollan \cite{AJKW}. Special cases of \Cnr{conj MM}, for $r=2$, and sparse graphs, have been proved by Gartland, Korhonen \& Lokshtanov \cite{GaKoLo} and by Hendrey, Norin, Steiner \& Turcotte \cite{HNST}.
The special case of \Cnr{Tho} for finitely presented Cayley graphs has been proved by MacManus \cite{McMFat}. The “only if” direction of \Prr{k planar} has been proved by Esperet \& Giocanti \cite{EsGiCoa}.

\bibliographystyle{plain}
\bibliography{../../collective}
\end{document}

%% file: defs.tex
\usepackage[usenames]{color} 
\usepackage{amsthm,amssymb,amsmath,bbm,enumerate,graphicx,epsf,stmaryrd,accents}
\usepackage[bookmarks, colorlinks=false, breaklinks=true]{hyperref} 

\usepackage{authblk}

\newcommand{\comment}[1]{}
\newcommand{\COMMENT}[1]{}

\definecolor{darkgray}{rgb}{0.3,0.3,0.3}
\newcommand{\defi}[1]{{\color{darkgray}\emph{#1}}}

\comment{
	\begin{lemma}\label{}	
\end{lemma}
\begin{proof}

\end{proof}

\begin{theorem}\label{}
\end{theorem} 
\begin{proof} 	

\end{proof}

}

\newtheorem{proposition}{Proposition}[section]
\newtheorem{definition}[proposition]{Definition}
\newtheorem{theorem}[proposition]{Theorem}
\newtheorem{corollary}[proposition]{Corollary}

\newtheorem{lemma}[proposition]{Lemma}
\newtheorem{observation}[proposition]{Observation}
\newtheorem{conjecture}{{Conjecture}}[section]

\newtheorem{problem}[conjecture]{{Problem}}

\newtheorem{question}[conjecture]{{Question}}

\newtheorem{examp}[proposition]{Example}

\newcommand{\FIG}{0}

\ifnum \NOTESON = 1 \newcommand{\note}[1]{ 

\hspace*{-30pt}
	{\color{blue}  NOTE: \color{Turquoise}{\small  \tt \begin{minipage}[c]{1.1\textwidth}  #1 \end{minipage} \ignorespacesafterend }} 
	
	}
\else \newcommand{\note}[1]{} \fi

\newcommand{\afsubm}[1]{ \ifnum \Debug = 1 {\mymargin{#1}}
\fi} 

\ifnum \Debug = 1 
\else  \fi

\ifnum \FIG = 1 \newcommand{\fig}[1]{Figure ``{#1}''}
\else \newcommand{\fig}[1]{Figure~\ref{#1}} \fi

\ifnum \FIG = 1 
\else  \fi

\ifnum \Debug = 1 \usepackage[notref,notcite]{showkeys}
\fi

\ifnum \COLORON = 0 \renewcommand{\color}[1]{}
\fi

\newcommand{\N}{\ensuremath{\mathbb N}}
\newcommand{\R}{\ensuremath{\mathbb R}}

\newcommand{\Z}{\ensuremath{\mathbb Z}}

\newcommand{\cb}{\ensuremath{\mathcal B}}
\newcommand{\cc}{\ensuremath{\mathcal C}}

\newcommand{\ch}{\ensuremath{\mathcal H}}

\newcommand{\cu}{\ensuremath{\mathcal U}}

\newcommand{\oo}{\ensuremath{\omega}}

\newcommand{\sm}{\backslash}

\makeatletter
\DeclareRobustCommand{\cev}[1]{%
  \mathpalette\do@cev{#1}%
}
\newcommand{\do@cev}[2]{%
  \fix@cev{#1}{+}%
  \reflectbox{$\m@th#1\vec{\reflectbox{$\fix@cev{#1}{-}\m@th#1#2\fix@cev{#1}{+}$}}$}%
  \fix@cev{#1}{-}%
}
\newcommand{\fix@cev}[2]{%
  \ifx#1\displaystyle
    \mkern#23mu
  \else
    \ifx#1\textstyle
      \mkern#23mu
    \else
      \ifx#1\scriptstyle
        \mkern#22mu
      \else
        \mkern#22mu
      \fi
    \fi
  \fi
}

\makeatother

\newcommand{\nin}{\ensuremath{{n\in\N}}}

\newcommand{\pth}[2]{\ensuremath{#1}\text{--}\ensuremath{#2}~path}

\newcommand{\pths}[2]{\ensuremath{#1}\text{--}\ensuremath{#2}~paths}

\newcommand{\seq}[1]{\ensuremath{(#1_n)_{n\in\N}}} 

\newcommand{\g}{\ensuremath{G\ }}
\newcommand{\G}{\ensuremath{G}}

\newcommand{\ceil}[1]{\ensuremath{\left\lceil #1 \right\rceil}}

\newcommand{\Cg}{Cayley graph}

\newcommand{\Lr}[1]{Lemma~\ref{#1}}

\newcommand{\Tr}[1]{Theorem~\ref{#1}}

\newcommand{\Sr}[1]{Section~\ref{#1}}

\newcommand{\Prr}[1]{Pro\-position~\ref{#1}}

\newcommand{\Cr}[1]{Corollary~\ref{#1}}
\newcommand{\Cnr}[1]{Con\-jecture~\ref{#1}}
\newcommand{\Or}[1]{Observation~\ref{#1}}

\newcommand{\Dr}[1]{De\-fi\-nition~\ref{#1}}

\newcommand{\lf}{locally finite}

\renewcommand{\iff}{if and only if}
\newcommand{\fe}{for every}
\newcommand{\Fe}{For every}

\newcommand{\st}{such that}

\newcommand{\ti}{there is}

\newcommand{\obda}{without loss of generality}

\newcommand{\wrt}{with respect to}

\newcommand{\istc}{is straightforward to check}

\newcommand{\leth}{large enough that}

\newcommand{\labequ}[2]{ \begin{equation} \label{#1} #2 \end{equation} } 
 
\newcommand{\labtequ}[2]{
 \begin{equation} \label{#1} 	\begin{minipage}[c]{0.9\textwidth}  #2 \end{minipage} \ignorespacesafterend \end{equation} }

\newcommand{\mymargin}[1]{
 \ifnum \Debug = 1
  \marginpar{%
    \begin{minipage}{\marginparwidth}\small%
      \begin{flushleft}%
        {\color{blue}#1}%
      \end{flushleft}%
   \end{minipage}%
  }%
 \fi
}%

\newcommand{\extras}[1]{
 \ifnum \Debug = 1
\section{Extras} #1
 \fi
}%

\newcommand{\mySection}[2]{}

\newcommand{\Erd}{Erd\H{o}s}












%% file: CoarseGTSurvey-Combinatorica-Revision_2.bbl
\begin{thebibliography}{10}

\bibitem{AHJKW}
S.~Albrechtsen, T.~Huynh, R.~W. Jacobs, P.~Knappe, and P.~Wollan.
\newblock {A Menger-Type Theorem for Two Induced Paths}.
\newblock {\em SIAM J.\ Discrete Math.}, 38(2):1438--1450, 2024.

\bibitem{AJKW}
S.~Albrechtsen, R.~Jacobs, P.~Knappe, and P.~Wollan.
\newblock A characterisation of graphs quasi-isometric to ${K}_4$-minor-free
  graphs.
\newblock arXiv:2408.15335.

\bibitem{BalMcMmet}
J.~Balig\'acs and J.~MacManus.
\newblock The metric {Menger} problem.
\newblock arXiv:2403.05630.

\bibitem{BelDraAsy}
G.~Bell and A.~Dranishnikov.
\newblock Asymptotic dimension.
\newblock {\em Topology and its Applications}, 155(12):1265--1296, 2008.

\bibitem{quasitrees}
I.~Benjamini and A.~Georgakopoulos.
\newblock Triangulations of uniform subquadratic growth are quasi-trees.
\newblock {\em {Annales Henri Lebesgue}}, 5:905--919, 2022.

\bibitem{BerSeyBou}
E.~Berger and P.~Seymour.
\newblock Bounded diameter tree-decompositions.
\newblock {\em Combinatorica}, 44(1):659--674, 2024.

\bibitem{BBEGLPS}
M.~Bonamy, N.~Bousquet, L.~Esperet, C.~Groenland, C.-H. Liu, F.~Pirot, and
  A.~Scott.
\newblock Asymptotic {Dimension} of {Minor}-{Closed} {Families} and
  {Assouad}-{Nagata} {Dimension} of {Surfaces}.
\newblock {\em J.\ Eur.\ Math.\ Soc.}, 26(10):3739--3791, 2011.

\bibitem{CDNRV}
V.~Chepoi, F.~F. Dragan, I.~Newman, Y.~Rabinovich, and Y.~Vax\`es.
\newblock Constant {Approximation} {Algorithms} for {Embedding} {Graph}
  {Metrics} into {Trees} and {Outerplanar} {Graphs}.
\newblock {\em Discrete \& Computational Geometry}, 47(1):187--214, 2012.

\bibitem{cornulier_metric_2016}
Yves Cornulier and Pierre de~la Harpe.
\newblock {\em Metric {Geometry} of {Locally} {Compact} {Groups}}.
\newblock EMS Press, 2016.

\bibitem{DHIM}
J.~Davies, R.~Hickingbotham, F.~Illingworth, and R.~McCarty.
\newblock Fat minors cannot be thinned (by quasi-isometries).
\newblock arXiv:2405.09383.

\bibitem{DiestelBook05}
Reinhard Diestel.
\newblock {\em Graph {T}heory \emph{(3rd edition)}}.
\newblock Springer-Verlag, 2005.
\newblock \\ Electronic edition available at:\\ {\small\tt
  http://www.math.uni-hamburg.de/home/diestel/books/graph.theory}.

\bibitem{DisPro}
M.~Distel.
\newblock Proper {Minor}-{Closed} {Classes} of {Graphs} have {Assouad}-{Nagata}
  {Dimension} 2.
\newblock {\em {arXiv:2308.10377}}, 2023.

\bibitem{DEMW}
V.~Dujmovi\'c, L.~Esperet, P.~Morin, and D.~R. Wood.
\newblock Proof of the {Clustered} {Hadwiger} {Conjecture}.
\newblock {\em {arXiv:2306.06224}}, 2023.

\bibitem{ErdPosInd}
P.~Erd\H{o}s and L.~P\'osa.
\newblock On {Independent} {Circuits} {Contained} in a {Graph}.
\newblock {\em Canadian Journal of Mathematics}, 17:347--352, 1965.

\bibitem{eriksson-bique}
S.~Eriksson-Bique, C.~Gartland, E.~Le~Donne, L.~Naples, and Nicolussi~G. S.
\newblock Nilpotent {Groups} and {Bi}-{Lipschitz} {Embeddings} {Into} {L1}.
\newblock {\em International Mathematics Research Notices},
  2023(12):10759--10797, 2023.

\bibitem{EsGiCoa}
L.~Esperet and U.~Giocanti.
\newblock Coarse geometry of quasi-transitive graphs beyond planarity.
\newblock arXiv:2312.08902.

\bibitem{FujPapCoa}
K.~Fujiwara and P.~Papasoglu.
\newblock A coarse-geometry characterization of cacti.
\newblock {arXiv:2305.08512}.

\bibitem{FujPapAsy}
K.~Fujiwara and P.~Papasoglu.
\newblock Asymptotic dimension of planes and planar graphs.
\newblock {\em Trans.\ Am.\ Math.\ Soc.}, 374:8887--8901, 2021.

\bibitem{GaKoLo}
P.~Gartland, T.~Korhonen, and D.~Lokshtanov.
\newblock On {Induced} {Versions} of {Menger}'s {Theorem} on {Sparse} {Graphs}.
\newblock arXiv:2309.08169.

\bibitem{Kleinian}
A.~Georgakopoulos.
\newblock {On planar Cayley graphs and Kleinian groups}.
\newblock {\em Trans.\ Am.\ Math.\ Soc.}, 373:4649--4684, 2020.

\bibitem{GrKrSiDec}
M.~Grohe, S.~Kreutzer, and S.~Siebertz.
\newblock Deciding first-order properties of nowhere dense graphs.
\newblock {\em {J.\ ACM}}, 64(3):1--32, 2017.

\bibitem{GroAsyInv}
M.~Gromov.
\newblock {Asymptotic invariants of infinite groups}.
\newblock In {\em {Geometric group theory, Vol.~2 (Sussex, 1991)}}, number 182
  in London Math.~Soc.~Lecture Note Ser., pages 1--295. Camb.\ Univ.~Press,
  1993.

\bibitem{halin65}
R.~Halin.
\newblock {\"U}ber die {M}aximalzahl fremder unendlicher {W}ege in {G}raphen.
\newblock {\em Math.\ Nachr.}, 30:63--85, 1965.

\bibitem{HNST}
K.~Hendrey, S.~Norin, R.~Steiner, and J.~Turcotte.
\newblock On an induced version of {Menger}'s theorem.
\newblock {arXiv:2309.07905}.

\bibitem{JorLanGeo}
M.~J\o{}rgensen and U.~Lang.
\newblock {Geodesic spaces of low Nagata dimension}.
\newblock {\em {Ann.\ Fenn.\ Math.}}, 47:83--88, 2022.

\bibitem{KerTre}
A.~Kerr.
\newblock Tree approximation in quasi-trees.
\newblock {\em {Groups Geom.~Dyn.}}, 17(4):1193--1233, 2023.

\bibitem{KhuCha}
A.~Khukhro.
\newblock {A Characterisation of Virtually Free Groups via Minor Exclusion}.
\newblock {\em Int.\ Math.\ Res.\ Not.\ IMRN}, page rnac184, 2022.

\bibitem{KroMolQua}
B.~Kr\"on and R.~G. M\"oller.
\newblock Quasi-isometries between graphs and trees.
\newblock {\em J.~Combin.\ Theory (Series B)}, 98(5):994--1013, 2008.

\bibitem{LiuAss}
C.-H. Liu.
\newblock Assouad-{Nagata} dimension of minor-closed metrics.
\newblock {arXiv:2308.12273}.

\bibitem{McMFat}
J.~MacManus.
\newblock Fat minors in finitely presented groups.
\newblock arXiv:2408.10748.

\bibitem{Manning}
J.~F. Manning.
\newblock Geometry of pseudocharacters.
\newblock {\em Geometry \& Topology}, 9:1147--1185, 2005.

\bibitem{NaorInt}
A.~Naor.
\newblock {Metric dimension reduction: A snapshot of the Ribe program}.
\newblock {\em Japan.\ J.\ Math.}, 7:167--233, 2012.

\bibitem{NgScSeCou}
T.~H. Nguyen, A.~Scott, and P.~Seymour.
\newblock A counterexample to the coarse {Menger} conjecture.
\newblock arXiv:2401.06685.

\bibitem{OstExp}
M.~I. Ostrovskii.
\newblock Expansion properties of metric spaces not admitting a coarse
  embedding into a hilbert space.
\newblock {\em {Comptes rendus de l'Academie Bulg.\ des Sci.}}, 62:415--420,
  2009.

\bibitem{OstRosMet}
M.~I. Ostrovskii and D~Rosenthal.
\newblock Metric dimensions of minor excluded graphs and minor exclusion in
  groups.
\newblock {\em Int.\ J.\ Algebra Comput.}, 25(4):541--554, 2015.

\bibitem{Ostrovskii}
Mikhail~I. Ostrovskii.
\newblock {\em {Metric Embeddings: Bilipschitz and Coarse Embeddings into
  Banach Spaces}}.
\newblock De Gruyter, 2013.

\bibitem{ThoHad}
C.~Thomassen.
\newblock The {Hadwiger} number of infinite vertex-transitive graphs.
\newblock {\em Combinatorica}, 12(4):481--491, 1992.

\bibitem{YuNov}
G.~Yu.
\newblock {The Novikov conjecture for groups with finite asymptotic dimension}.
\newblock {\em Ann.\ of Math.}, 147(2):325--355, 1998.

\end{thebibliography}
